%% file: mainPhi-n.tex
\documentclass[12pt,twoside,a4paper]{article}
\usepackage{authblk}
\usepackage{titling}
\usepackage{titletoc}
\usepackage{url}
\usepackage[utf8]{inputenc}
\usepackage{amsfonts}
\usepackage{amssymb}
\usepackage{amsthm}
\usepackage[pdftex]{graphicx}
\usepackage[bookmarksopen=false,pdftex=true,breaklinks=true,%
      backref=page,pagebackref=true,plainpages=false,%
      hyperindex=true,pdfstartview=FitH,%
      pdfpagelabels=true,colorlinks=true,linkcolor=blue,%
      citecolor=red,urlcolor=red,hypertexnames=false,colorlinks=true%
      ]%
   {hyperref}
\usepackage{color}
\usepackage{mathrsfs}
\usepackage{mathtools}
\usepackage[french,english]{babel}\frenchsetup{GlobalLayoutFrench=false}
\usepackage{natbib}
\bibpunct{(}{)}{,}{a}{}{,}
\usepackage{xspace}

\usepackage{etoolbox}

\makeatletter\def\th@plain{\slshape}\makeatother
\makeatletter\patchcmd{\th@remark}{\itshape}{\slshape}{}{}\makeatother

\newcounter{bidon}

\newcommand{\rdb}{\refstepcounter{bidon}}

\input{MathMacrosPhi-n.tex}

\usepackage[bookmarksopen=false,breaklinks=true,%
      backref=page,pagebackref=true,plainpages=false,%
      hyperindex=true,pdfstartview=FitH,colorlinks=true,%
      pdfpagelabels=true,linkcolor=blue,%
      citecolor=red,urlcolor=red,%hypertexnames=false%
      ]%
   {hyperref}

\begin{document}
\selectlanguage{english}

\thispagestyle{empty}
~ 
\vspace{3cm}

\noindent In this file you find the English version starting on the page  numbered \pageref{beginenglish}.

\medskip \noindent  {\large \bf Cyclotomic polynomials without using the zeros of $Y^n-1$}

\bigskip \noindent  
Then the French version begins on the page numbered \pageref{beginfrench}.

\medskip\noindent   {\large \bf Polynômes cyclotomiques, sans utiliser les zéros des polynômes $Y^n-1$}

\smallskip \noindent Le lecteur ou la lectrice sera sans doute surprise de l'alternance des sexes ainsi que de l'orthographe du mot 'corolaire', avec d'autres innovations auxquelles elle n'est pas habituée. En fait, nous avons essayé de suivre au plus près les préconisations de l'orthographe nouvelle recommandée, telle qu'elle est enseignée aujourd'hui dans les écoles en France.  

\bigskip\noindent   {\large \bf Authors}  

\smallskip \noindent   Gema M. Diaz-Toca, Universidad de Murcia, Spain,
email: {\tt gemadiaz@um.es}
\\
Partially supported by the grant\\  PID2020-113192GB-I00/AEI/10.13039/501100011033 (Mathematical Visualization: Foundations, Algorithms and Applications) from the Spanish State Research Agency (Ministerio de Ciencia e Innovaci\'{o}n

\smallskip \noindent Henri Lombardi, Université Marie et Louis Pasteur, F-25030 Besançon Cedex, France, \\
email: {\tt henri.lombardi@univ-fcomte.fr}

\smallskip \noindent Claude Quitté, Laboratoire de Math\'ematiques,
SP2MI, Boulevard 3, Teleport 2, BP 179, F-86960 Futuroscope Cedex,
France,\\ 
 email: {\tt claude.quitte@math.univ-poitiers.fr}

\normalsize
\newpage
\thispagestyle{empty}

~

\pagestyle{headings}
\patchcmd{\sectionmark}{\MakeUppercase}{}{}{}
\setcounter{page}{0}
\renewcommand\thepage{E\arabic{page}}
\input{englishPhi-n.tex}

\clearpage
\newpage
\thispagestyle{empty}

%: peut être ajouter ici une page blanche pour que la version française commence en pdf impair
~

\clearpage
\newpage
\thispagestyle{empty}

\renewcommand\thepage{F\arabic{page}}\renewcommand\theHsection{F\arabic{section}}
\input{frenchPhi-n.tex}

%% file: MathMacrosPhi-n.tex
\newcommand {\junk}[1]{}

\newcommand \noi {\noindent}

\renewcommand \leq{\leqslant}

\renewcommand \geq{\geqslant}

\newcommand \etl{^*}

\newcommand \etoz{$^*$}

\newcommand \divi {\mid}
\newcommand \nedivi {\not\kern 2.5pt\mid}

\newcommand{\pref}[1]{\textup{\hbox{\normalfont(\ref{#1})}}}

\newcommand \aqo[2] {#1\sur{\gen{#2}}\!}

\newcommand \gen[1] {\left\langle{#1}\right\rangle}

\newcommand \sur[1] {\!\left/#1\right.}

\renewcommand \leq{\leqslant}
\renewcommand \geq{\geqslant}

\newcommand \som {\sum\nolimits}

\newcommand\tsbf[1]{\textsf{\textbf{\textup{#1}}}}

\newcommand\Tsbf[1]{\hyperref[Ax#1]{\tsbf{#1}}}

\newcommand\Sa[1]{\hyperref[theorie#1]{\sa{#1}}}
\newcommand\sa[1]{\hbox{\usefont{T1}{pzc}{m}{it}#1}\,}

\newcommand \eoe {\hbox{}\nobreak\hfill
\vrule width .5em height .5em depth 0mm \par \smallskip}

\newcommand \ov[1] {\overline{#1}}

\makeatletter
\def\revddots{\mathinner{\mkern1mu\raise\p@
\vbox{\kern7\p@\hbox{.}}\mkern2mu
\raise4\p@\hbox{.}\mkern2mu\raise7\p@\hbox{.}\mkern1mu}}
\makeatother

\newcommand \CC{\mathbb {C}}
\newcommand \FF{\mathbb {F}}

\newcommand \NN{\mathbb {N}}
\newcommand \ZZ{\mathbb {Z}}

\newcommand \PP{\mathbb {P}}
\newcommand \QQ{\mathbb {Q}}
\newcommand \RR{\mathbb {R}}

\newcommand \Fp{{\FF_p}}

\newcommand \Q{\mathbb {Q}}
\newcommand \Z{\mathbb {Z}}

\newcommand \gk {\mathbf{k}}

\newcommand \gl {\mathbf{l}}
\newcommand \gA {\mathbf{A}}
\newcommand \gB {\mathbf{B}}
\newcommand \gC {\mathbf{C}}

\newcommand \gK {\mathbf{K}}
\newcommand \gKb {\ov\gK}

\newcommand \gL {\mathbf{L}}

\newcommand \gR {\mathbf{R}}

\newcommand \gV {\mathbf{V}}

\newcommand \gZ {\mathbf{Z}}

\newdimen\xyrowsp
\newcommand{\SCO}[6]{
\xymatrix @R = \xyrowsp {
                                  &1 \ar@{-}[dl] \ar@{-}[dr] \\
#3 \ar@{-}[ddr]                   &   & #6 \ar@{-}[ddl] \\
                                  &\bullet\ar@{-}[d] \\
                                  &\bullet   \\
#2 \ar@{-}[ddr] \ar@{-}[uur]      &   & #5 \ar@{-}[ddl] \ar@{-}[uul] \\
                                  &\bullet \ar@{-}[d] \\
                                  &\bullet  \\
#1 \ar@{-}[uur]                   &   & #4 \ar@{-}[uul] \\
                                  & 0 \ar@{-}[ul] \ar@{-}[ur] \\
}
}

\renewcommand \deg {\MA{\mathrm{deg}}}

\newcommand \pgcd {\MA{\mathrm{pgcd}}}
\newcommand \ppcm {\MA{\mathrm{ppcm}}}
\newcommand\MA[1]{\mathop{#1}\nolimits}

\newcommand \lcm {\mathop\mathrm{lcm}}

\newcommand \rG {\mathrm{G}}

\newcommand \vu {\vee} % sup
\newcommand \vi {\wedge} % inf
\newcommand \AY {\gA[Y]}

\newcommand \KX {\gK[X]}
\newcommand \KY {\gK[Y]}

\newcommand \QQy {\QQ[y]}

\newcommand \QQY {\QQ[Y]}
\newcommand \QX {\QQ[X]}
\newcommand \QY {\QQ[Y]}
\newcommand \Qn {\Q_n}

\newcommand \ZX {\ZZ[X]}
\newcommand \ZZY {\ZZ[Y]}
\newcommand \ZY {\ZZ[Y]}

%% file: englishPhi-n.tex
%!TEX encoding =  UTF-8 Unicode
%!TEX root =  mainPhi-n.tex

\begingroup

\def\proofname{\textsl{Proof}}

\title{Cyclotomic polynomials without using the zeros of $Y^n-1$}
\author{Gema M. Diaz-Toca, Henri Lombardi, Claude Quitt\'e}

%\date{à choisir}
\maketitle

\input{EnglishTheoremsPhi-n.tex}

\input{EnglishMacrosPhi-n.tex}
\rdb
\label{beginenglish}

\begin{abstract}
This note aims to construct an \gui{intrinsic} splitting field for the polynomial $Y^n-1$ over the rational field $\QQ$, in a way that Gauss, Kummer, Kronecker and Bishop would have liked. Contrary to the usual presentations, our construction does not use any splitting field of $Y^n-1$ which would be given before demonstrating the irreducibility of the \cyct \pol.
\end{abstract}

%\medskip \noindent {\bf Keywords:} 

%\smallskip \noindent {\bf MSC:} 

%\newpage

\section*{Introduction} 

The most widespread definition of the \textsl{$n$-th cyclotomic polynomial}  $\Phi_n(Y)$ explicitly makes use of the roots of the unity in $\CC$,
\begin{equation}\label{def1}
\Phi_n(Y)=\prod\nolimits_\zeta(Y-\zeta),
\end{equation}
where $\zeta$ runs over the set of primitive $n$-th roots of of unity, i.e.\ $\zeta^n=1$ and $\zeta^d\neq 1$ if $d\divi n$ and $d\neq n$.
 
Gauss proved that $\Phi_n(Y)\in\ZZY$ and also that $\Phi_p$ is irreducible in $\QQY$ when $p$ is prime. Kronecker proved that $\Phi_n(Y)$ is irreducible for an arbitrary $n$.

The \textsl{$n$-th inverse cyclotomic polynomial} is defined as $\Psi_n(Y)=\frac{Y^n-1}{\Phi_n(Y)}$.

Another well known definition of cyclotomic polynomials is given inductively by  $\Phi_1(Y)= Y-1$ and 
\begin{equation}\label{def2}
\Phi_n(Y)=\frac{Y^n-1}{\prod\limits_{0<d<n,d\mid n}\Phi_d(Y)},\,\,\, n\geq 2.
\end{equation} 
For this second definition by induction, it is necessary to prove first that \smash {$\prod\limits_{0<d<n,d\mid n}\Phi_d(Y)$} divides $Y^n-1$. This can be deduced from the first definition. 

All known irreducibility proofs of $\Phi_n(Y)$ in $\QQ[Y]$ use a splitting field of $Y^n-1$. 

\smallskip In this note we propose a new direct definition as follows. For $n=1$, $\Psi_1(Y)=1$ and  $\Phi_1(Y)=Y-1$. For $n\geq 2$,
\begin{equation}\label{def3}
\Psi_n(Y)=\lcm\limits_{d:0<d<n,d\mid n}(Y^d-1)  \; \text{ and }  \;\Phi_n(Y)=(Y^n-1)/\Psi_n(Y).
\end{equation} 

Then we demonstrate the main properties of cyclotomic polynomials, in particular that $\Phi_n(Y)$ is irreducible for every $n$. 

So we get an \textsl{intrinsic} splitting field of $Y^n-1$ as $\Qn:=\aqo{\QQY}{\Phi_n}$
 in a way that Gauss, Kummer, Kronecker, and Bishop would have liked.

\smallskip 
This note is written in a fully constructive manner so that the results of the proofs can be directly implemented in a formal proof management system. In particular, quotient structures are addressed through congruence systems, reflecting the spirit of Gauss, Kummer, and Kronecker, renewed by Bishop in his minimalist constructive version of set theory (\cite{Bi67}).   
 
\smallskip 
It is important to clarify that we do not intend to present a new algorithm for the efficient calculation of cyclotomic polynomials. To our knowledge, the algorithms introduced in reference \cite{ArMo} remain the most effective for this purpose. Even so, we did some calculations using SAGE 10.4, Mathematica 13, and Maple 2022. In SAGE, we found that the algorithm defined in our paper is surprisingly more efficient than the native one when $n $ is a large product of distinct prime numbers. Furthermore, we observe that the execution times in Maple are unmatched, and it is probably because the algorithm used in Maple is the SPL algorithm  introduced in \cite{ArMo}.

\smallskip For further reading on cyclotomic polynomials, we recommend the works of \cite{Bro,Kro,Wei,Tignol}.
 
\section{Preliminaries}
Below we introduce some notation and classic results of polynomial theory, which we will use in Section \ref{main}. 

Let $\PP$ denote the set of prime numbers. 

The \textsl{\cara Euler function} ie the map $\phi\colon\NN\etl\to\NN\etl$ where $\phi(n)$ is the number of generators of a cyclic group with $n$ \elts. 

In a gcd domain,   
 the $\gcd(a,b)$ will be denoted by $a\vi b$, and the $\lcm(a,b)$ ($a,b\neq 0$) by $a\vu b$.  
 
Given a polynomial $f$ in $\ZY$, the greatest common divisor of its coefficients is defined to be the content of $f$ and denoted by $\rG(f)$. If the content is equal to 1, the polynomial is said to be primitive.
We now recall the relation between divisibility in $\ZY$ and that in $\QY$.\footnote{Classic Proposition \ref{propZXQX} is generalised by replacing $\Z$ with a gcd domain and $\QQ$ with its quotient field.} 
 
%p
%:     Proposition{propZXQX}
\begin{proposition}[$\ZY$ is a gcd domain] \label{propZXQX}~
\begin{enumerate}
\item  Let $f,g\in\QY$ be monic polynomials. If $fg\in\ZY$, then $f$ and $g$ are in $\ZY$ (Gauss Lemma).
\item Let $f,g\in\ZY$. Then $f$ divides $g$ in $\ZY$ \ssi $f$ divides $g$ in $\QY$ and $\rG(f)$ divides $\rG(g)$ in $\Z$.
\item The ring $\ZY$ is a gcd domain. A \pol $f\in\ZY$ is irreducible \ssi it is primitive and irreducible in $\QY$. 
\item  Two primitive polynomials of $\ZY$ are coprime \ssi they are coprime in $\QY$.
\end{enumerate} 
\end{proposition}
%----------- fin proposition ----------------------------- 

Recall that a field is said to be discrete if it has a zero test, which in particular allows us to treat linear algebra on the field in a completely explicit way.

%:     Fact{fact1}
\begin{lemma} \label{fact1}
In an arbitrary discrete field $\gK$, we have in $\KY$
\[(Y^m-1)\vi (Y^n-1)=Y^{ m\vi n}-1,
\]
for $m,n \in \mathbb{N}$.
\end{lemma}
%----------- fin fact ----------------------------------------- 

This is based on Euclidean algorithm and on the fact that if $r$ is the remainder of the division of $n\geq 2$ by $m\geq 2$, then $Y^r-1$ is the remainder of the division of $Y^n-1$ by $Y^m-1$.

%f
%:     Fact{fact2} lemme de substitution
\begin{lemma}[Substitution Lemma] \label{fact2}~\\ 
Let $\gK$ be a discrete field and let $f_1,\dots,f_\ell,h\in\KY$ be non-constant monic \pols. Let $g=\gcd(f_1,\dots,f_\ell)$ and  $f=\lcm(f_1,\dots,f_\ell)$, $g$ and $f$ monic. Let $F_i(Y)=f_i(h(Y))$, $G(Y)=g(h(Y))$ and $F(Y)=f(h(Y))$. Then $G$ is the greatest common divisor of $F_1, \dots,F_\ell$ and $F$ is their  least common multiple.
\end{lemma}
%----------- fin fact ----------------------------------------- 
%
\begin{proof} 

The substitution of $h(Y)$ for $Y$ is a homomorphism of \Klgs, from $\KY$ into itself because it is an evaluation homomorphism. The point here is to see that it is also a morphism for the laws of gcd and lcm. This is also true for any morphism between Bézout domain, but with the ambiguity due to the fact that the gcd and the lcm are only well defined up to association. In the present case, we benefit from the fact that every non-zero polynomial is associated with a unique monic polynomial.

Given the associativity of the gcd and the lcm, it is enough to treat the case where $\ell=2$.
In the Bézout relations which express $g$ as the linear combination of $(f_1,f_2)$, we substitute $h(X)$ for $X$ and we obtain the Bézout relations which certify that $G$ is the gcd of $(F_1,F_2)$. For the lcm, we substitute $h(Y)$ for $Y$ in the equality $g(Y)f(Y)=f_1(Y)f_2(Y)$.
\end{proof}

%f
%:     Fact{fact3}
\begin{lemma}[Roots and Decomposition] \label{fact3} 

Let $\gA$ be a ring and let $f\in \AY$ be a monic polynomial, $\deg(f)=n$. Let $x_1,\dots,x_m$ be zeros of $f$ in $\gA$, with $x_i-x_j$ invertible in~$\gA$ for $i\neq j$. Then we have that
\begin{enumerate}
\item $f$ is divisible by $\prod_{i=1}^m(Y-x_i)$,
\item if $m=n$, then $f=\prod_{i=1}^n(Y-x_i)$,
\item if $m>n$, the ring is trivial ($1=_{\gA}0$).
\end{enumerate}
\end{lemma}
%----------- fin fact ----------------------------------------- 
\facile

\section{Construction of \cycts \pols}\label{main}

Next, we provide an algebraic and constructive definition of both, the cyclotomic polynomial and the inverse cyclotomic polynomial.  
% (see \cite{ArMo}, \cite{Mor}).

%d
%:     Definota{defiPolcyc}
\begin{definota} \label{defiPolcyc}~\\
We define the $n$-th cyclotomic polynomial $\Phi_n(Y)$ and the $n$-th inverse cyclotomic polynomial $\Psi_n(Y)$ in $\QQY$ as follows:

\begin{itemize}
\item For $n=1$, $\Psi_1(Y)=1$ and $\Phi_1(Y)=Y-1$,
\item For $n\geq 2$, $\Psi_n(Y)=\lcm\limits_{d:0<d<n,d\mid n}(Y^d-1)$  and  $\Phi_n(Y)=(Y^n-1)/\Psi_n(Y)$. 
%
%\item  
\end{itemize}
We also write  $\Gamma_n(Y)=(Y^n-1)/(Y-1)$, thus, {$\smash\Psi_n(Y)=(Y-1)\lcm\limits_{d:1<d<n,d\mid n}\Gamma_d(Y)$}.\\
Finally, let $\Qn:=\aqo{\Q[Y]}{\Phi_n(Y)}$ and
$y_n$ be $Y$ modulo $\Phi_n$ in $\QQY$. So $\Qn=\QQ[y_n]$. 
\end{definota}
%----------- fin definition --------------------------------
Note that we have not claimed at the moment that $\Qn$ is a field.

\smallskip 
Observe that the quotient $\Phi_n $ is a well-defined element in $\QQY$. Indeed, each polynomial $Y^d-1$ divides $Y^n-1$ when $ d$ is a divisor of $ n$. Consequently, this property holds for their least common multiple $\Psi_n $ as well. Since we treat the lcm as monic polynomial, $\Phi_n $ is monic and well-defined.

\smallskip 
For example, we have
\begin{itemize}
  \item[-]$\Phi_2(Y)=Y+1 ,  \Phi_3(Y)=Y^2+Y+1,$
  \item[-]$\Psi_4(Y)=Y^2-1 , \Phi_4(Y)=Y^2+1,$
  \item[-]$\Psi_8(Y)=Y^4-1 ,  \Phi_8(Y)=Y^4+1.$
   \item[-]For $p\in\PP$, $\Psi_p=\Phi_1$ so that $\Phi_1\Phi_p=Y^p-1$, and $\Phi_p=\Gamma_p=\som_{r=0}^{p-1}Y^{r}$. 

\end{itemize}

%l
%:     Lemma{lemPolcyc1}
\begin{lemma} \label{lemPolcyc1}
For every $n$, the monic \pol $\Phi_n(Y)$ has coefficients in $\ZZ$. 
\end{lemma}
%----------- fin lemma ----------------------------------- 
%
\begin{proof}
Point \textsl{1}, Proposition \ref{propZXQX}. 
\end{proof}
%
%%l
%%:     Lemma{lemPolycyc01}
%\begin{lemma} \label{lemPolycyc01}
%Le \pol $\Phi_n$ est étranger à $\Psi_n$, et donc aux \pols  $Y^d-1$ pour $d$ divisant strictement $n$ ainsi qu'à leurs diviseurs. 
%\end{lemma}
%%----------- fin lemma ----------------------------------- 
%%
%\begin{proof}
%En effet, $Y^n-1$ est \spl et $Y^n-1=\Phi_n\Psi_n$.
%\end{proof}
%

%:     Lemma{lemPolcyc2}
\begin{lemma} \label{lemPolcyc2}
We have
 $\Psi_n(Y)=\lcm_{p:p\in\PP,p\mid n}(Y^{n/p}-1)$.
\end{lemma}
%----------- fin lemma ----------------------------------- 
%
\begin{proof}
If $d$ strictly divides $n$, then it divides one of $n/p$ for a $p\in\PP$ which divides $n$. Thus,
 $Y^d-1$ divides $Y^{n/p}-1$ for this prime number $p$.
\end{proof}
 The following two results are well known, but note that their demonstrations do not mention the roots of unity.

%The two following corollaries are surpassed by the results of the next section.

%:     Corollary{coremPolcyc2}
\begin{corollary} \label{coremPolcyc2}
If $p\in\PP$ divides $n$, we have $\Psi_{np}(Y)=\Psi_{n}(Y^p)$ and $\Phi_{np}(Y)=\Phi_{n}(Y^p)$.\footnote{Regarding the degrees, it corresponds to $\phi(np)=p\phi(n)$.} 
\end{corollary}
%--------- fin corollary ------------------------------- 
%
\begin{proof}
Since $p$ divides $n$, a prime number $q$ divides $np$ \ssi $q$ divides $n$, so
\[
\Psi_{np}(Y)=\lcm_{q:q\in\PP,q\mid n}(Y^{p {n/q}}-1),
\]
which is obtained by substituting $Y^p$ for $Y$ in $\Psi_{n}(Y)=\lcm_{q:q\in\PP,q\mid n}(Y^{n/q}-1)$. By the substitution lemma, we have $\Psi_{np}(Y)=\Psi_{n}(Y^p)$. Consequently, $\Phi_{np}(Y)$ is obtained by substituting $Y^p$ for $Y$ in $(Y^n-1)/\Psi_{n}(Y)$.
 \end{proof}
%
%c
%:     Corollary{corcoremPolcyc2}
\begin{corollary} \label{corcoremPolcyc2}
If every prime factor of a number $r$ divides $n$, then $\Phi_{nr}(Y)=\Phi_{n}(Y^r)$.
\end{corollary}
%--------- fin corollary ------------------------------- 
%
\begin{proof} 
Results from the previous corollary using the decomposition of $r$ into prime factors.
\end{proof}

Thus, computing  $\Phi_n$ is thus reduced to the case where $n$ is an integer without square factors, product of distinct prime numbers. For example, with $n=360$,  we have $\Phi_{360}(Y)=\Phi_{180}(Y^2)=\Phi_{90}(Y^4)=\Phi_{30}(Y^{12})$.

\section{Crucial properties}

Once introduced the definition, the most important properties of cyclotomic polynomials will be demonstrated, without using the primitive roots of unity, using only our definition.

%l
%:     Lemma{lemcrucial}
\begin{lemma} \label{lemcrucial} Let $d,e\geq 2$. Over $\ZY$, we have then:
\begin{enumerate}
\item if $d \divi e$, then $\Phi_d\vi (Y^e-1) =\Phi_d$,
\item if $d \nedivi e$, then $\Phi_d\vi (Y^e-1) =1$,
\item if $d\neq e$, then $\Phi_d\vi\Phi_e=1$.
\end{enumerate}
\end{lemma}
\begin{proof}
\textsl{1}. If $d\divi e$, then $\Phi_d$ divides $Y^d-1$, which divides $Y^e-1$.

\smallskip \noindent  
\textsl{2}. If $d\nedivi e$, then $d\vi e$ is a strict divisor of $d$.\\
Now $\Phi_d\vi (Y^e-1)$ divides $(Y^d-1)\vi(Y^e-1)=Y^{d\vi e}-1$, which divides $\Psi_d$ since $d\vi e$ is a strict divisor of $d$. 
%Or $\Phi_d\vi\Psi_d=1$, 
So, $\Phi_d\vi (Y^e-1)$ divides both $\Phi_d$ and $\Psi_d$. Finally, since $\Phi_d\vi \Psi_d=1$ ($Y^n-1$ is separable),  $\Phi_d\vi (Y^e-1)$ must be equal to $1$.

\smallskip \noindent 
\textsl{3}.  One of the two does not divide the other. For example,
$d\nedivi e$. Then, according to point \textsl{2}, we have~$\Phi_d\vi(Y^e-1)=1$. Since $\Phi_e\divi Y^e-1$, we have $\Phi_d\vi\Phi_e=1$. 
\end{proof}

We now prove that our \dfn agrees with the second definition \pref{def2}.

%t
%:     Theorem{Mainth1}
\begin{theorem} \label{Mainth1}
Given an integer $n\geq 1$, the polynomial $Y^n-1$ in $\Z[Y]$ factors into pairwise coprime factors as
\[Y^n-1=\prod\limits_{d:d\mid n}\Phi_d(Y). \eqno{(1_n)}
\]
In particular, $\deg(\Phi_n)=\phi(n)$ (by induction, since $n=\som\limits_{d:d\mid n}\phi(d)$.)
\end{theorem}
%----------- fin theorem ----------------------------- 
%
\begin{proof} 
We will use induction on $n$. When $n=1$ or $n$ is a prime number, the factorization $(1_n)$ is obvious. Now assume that the \prt holds for all the strict divisors $d$ of $n$. We then observe that $\Psi_n$ is the lcm of all $\Phi_d$, $d$ strict divisor of $n$.  Indeed, by the induction assumption, for every divisor $d$, the polynomial $Y^d-1$ is product of all the polynomials $\Phi_h$, $h$ divisor of $d$, and it is also their lcm because $Y^d-1$ is \spl.

Thus we have
\[
Y^n-1=\Psi_n(Y)\,\Phi_n(Y)=
\bigg(\lcm\limits_{d:d\mid n,1\leq d<n}(Y^d-1)\bigg)\,\Phi_n(Y)=
\bigg(\lcm\limits_{d:d\mid n,1\leq d<n}\Phi_d(Y)\bigg)
\,\Phi_n(Y). 
\]
In the first factor of the last product, the polynomials are pairwise coprime 
according to Lemma \ref{lemcrucial}, so their lcm is equal to their product, which completes the proof.

Observe that $\Phi_n$ is also coprime to   $\Phi_d(Y)$ when $d$ divides strictly $n$ because  $Y^n-1$ is \spl.  
\end{proof}

Undoubtedly, one of the most famous results is the irreducibility of $\Phi_n$; see various proofs in \cite{Wei}. Here we provide one more.

%
%:     Theorem{Mainth2}
\begin{theorem} \label{Mainth2}
Given $n\geq 1$, the polynomial $\Phi_n$ is irreducible over $\QQY$.
\end{theorem}
%----------- fin theorem ----------------------------- 

%
\begin{proof}[Proof - Beginning] Obviously, $\Phi_1$ is irreducible.  
Let $n\geq 2$ and let $f$ be a monic \pol of degree $\geq 1$ which divides $\Phi_n$ in $\QQY$. The theorem will follow once we prove that $f=\Phi_n$.
Observe that $f\in\ZZY$ and it is enough to prove that $\deg(f)\geq\phi(n)$. We consider the \QQlg $\QQ_f=\aqo\QQY{f}$. Let $y$ be $Y$ modulo $f$, hence $\QQ_f=\QQy$. We will begin by establishing the following two lemmas.% \end{proof}
%
%l
%:     Lemma{lem1Mainth2}
\begin{lemma} \label{lem1Mainth2}
We have $Y^n-1=\prod\limits_{0\leq r<n}(Y-y^r)$ in $\QQ_f[Y]$. 
\end{lemma}
%----------- fin lemma ----------------------------------- 
%
\begin{proof} 

Since $f$ divides $Y^n-1$, we have $y^n=1$. So the $y^r$ are zeros of $Y^n-1$. Hence, according to Lemma \ref{fact3}, it suffices to see that the $y^r-y^s$, $n>r>s$, are invertible in $\QQ_f$. Let us write $r=s+t$ and $y^r-y^s=y^s(y^t-1)$. Since $n$ does not divide~$t$ and $f$ divides $\Phi_n$, we have $(Y^t-1)\vi f(Y)=1$ by point \textsl{2} of Lemma \ref{lemcrucial}. The Bézout relation in $\QQ_f[Y]$ for $(Y^t-1)\vi f(Y)=1$, specialized in $y$, shows that \hbox{ $y^t-1$} is invertible in $\QQ_f$.  
\end{proof}
%

%:     Lemma{lem2Mainth2}
\begin{lemma} \label{lem2Mainth2}
If $r\vi n=1$ and $n\geq 2$, then $f(Y)$ divides $f(Y^r)$ in $\QQY$ (or in $\ZZY$).% $\QQ_f[Y]$. 
\end{lemma}
%----------- fin lemma ----------------------------------- 
%
\begin{proof} 
Observe that if $f(Y)$ divides $f(Y^r)$, then, by substitution $f(Y^k)$ divides $f(Y^{kr})$ for all $k$. Therefore, to prove the result, it suffices to consider the case $r=p$, a prime number which does not divide $n$.

We write $Y^n-1 = f(Y)g(Y)$ in $\ZY$, so $Y^{np}-1 = f(Y^p)g(Y^p)$. Since $(Y^n-1) \mid (Y^{np}-1)$, we have $f(Y) \mid f(Y^p)g(Y^p)$. In $\Fp[Y]$, since $p\nedivi n$, we have $\ov f(Y)\vi \ov g(Y) = \ov1$ and so $\ov f(Y)\vi \ov g(Y)^p = \ov1$. Moreover, $\ov g(Y)^p = \ov g(Y^p)$, which implies $\ov f(Y) \vi \ov g(Y^p) = \ov 1$, and a fortiori $f(Y) \vi g(Y^p) = 1$ in $\ZZY$. This shows that $f(Y)\mid f(Y^p)$ and the result is proved.
\end{proof}
%
%\noindent \textsl{Remark}.  More precisely, the proof actually says that if a monic polynomial  $F\in \ZZY$ divides $Y^n-1$ and if $r\vi n = 1$, then $F(Y)$ divides $F(Y^r)$ in $\QQY$. \eoe

%
\noindent \textsl{End of proof of  \thref{Mainth2}.} 
According to the previous lemma, by specializing $Y$ in~$y$, we obtain that $f(y^r)=0$ for
$r<n$, coprime with $n$. But the $y^r-y^{r'}$ are invertible in~$\QQ_f$ (already proved), so $f(Y)$ is multiple of $\prod\limits_{r:1\leq r<n,r\vi n=1}(Y-y^r)$
by Lemma \ref{fact3}. Thus, $\deg(f)\geq \phi(n)$ and it follows that $f=\Phi_n$, $y=y_n$, $\Qn=\QQ_f$ and
\[f(Y)=\Phi_n(Y)=\prod\limits_{r:1\leq r<n,r\vi n=1}(Y-y_n^r).
\] 

\end{proof}

\noindent \textsl{Remark}.  More precisely. 
\begin{itemize}
\item The same proof as Lemma \ref{lem1Mainth2} provides $Y^n-1=\prod\limits_{0\leq r<n}(Y-y_n^r)$ in $\Qn[Y]$.

\item The proof of Lemma \ref{lem2Mainth2} actually says that if a monic polynomial  $F\in \ZZY$ divides $Y^n-1$ and if $r\vi n = 1$, then $F(Y)$ divides $F(Y^r)$ in $\QQY$. \eoe
\end{itemize}

Finally, the following corollary connects our definition with the classical definition of cyclotomic polynomials.
%:     Corollary{corMainths0}
\begin{corollary} \label{corMainths0}
Let $\gK$ be an extension field of $\QQ$ which contains a primitive $n$-th root of unity, denoted $\xi$. Then the subfield $\QQ[\xi]$ is isomorphic to $\Qn$. Such an isomorphism $\varphi\colon \QQ[\xi] \to \Qn$ is determined by the image of $\xi$, which is any $n$-th primitive root of unity in $\Qn$.
\end{corollary}
%--------- fin corollary ------------------------------- 
%
\begin{proof} 
By hypothesis, $\xi$ satisfies $Y^n-1=0$ but it is not root of any $Y^d-1$,  
if $d$ strictly divides $n$. In particular, $\Phi_d(\xi)$ is invertible in $\gK$.
The equality $\xi^n-1=\prod_{d \mid n}\Phi_d(\xi)$ then shows that $\Phi_n(\xi)=0$. Since $\Phi_n(Y)$ is irreducible in $\QQ$, it is the minimal polynomial of $\xi$ over $\QQ$.
\end{proof}
%

%c
%:     Corollary{corMainths}
\begin{corollary} \label{corMainths} Let $y_n$ be $Y$ modulo $\Phi_n$. 
If $n$ is even, then the order of the group of roots of unity of $\Qn$ is equal to $n$, generated by $y_n$. If $n$ is odd, then the order of the group of roots of unity of $\Qn$ is equal to $2n$, generated
by $-y_n$. 
\end{corollary}
%--------- fin corollary ------------------------------- 
%
\begin{proof}
Suppose first that $n$ is even. If the group of roots of unity were strictly larger, it would contain an element $y$ of order $nk$, $k>1$. This $y$ would generate a subfield isomorphic to $\Q_{nk}$ (Corollary \ref{corMainths0}) of dimension $\phi(nk)$. Since $n$ is even, we would have $\phi(nk)>\phi(n)$.

If $n$ is odd, the order of $-y_n$ is $2n$ in $\Qn$, so $\Qn=\QQ_{2n}$ and the group of roots of unity cannot be of order greater that $2$n according to the even case. 
\end{proof}
%Noticed. It is strange that we have to wait so long before demonstrating this intuitive result.

%

%:     corollary{corMainths3}
\begin{corollary} \label{corMainths3}
Let $n$ be an even integer. The field $\Qn$ contains a subfield isomorphic to~$\Q_d$ \ssi $d$ divides $n$.  
\end{corollary}
%----------- fin cor ----------------------------------- 
%:     corollary{corMainths4}
\begin{corollary} \label{corMainths4}
If $n$ is odd $\geq 3$, then $\Phi_{2n}(Y)= \Phi_{n}(-Y)$. 
\end{corollary}
%----------- fin cor ----------------------------------- 
%
\medskip \noindent 
{\bf Conclusion.} The ``intrinsic'' approach we propose seems more elementary than the usual, more learned approach that uses the existence of a splitting field for the polynomial $Y^n-1$. However, although it is quite easy to prove that any group of roots of unity in a field is necessarily cyclic, we regret not having been able to exploit this fact to radically simplify the proof of the irreducibility of $\Phi_n$.

\bibliographystyle{plainnat}
%\bibliography{Phinbib.bib} 

\endgroup

%% file: EnglishTheoremsPhi-n.tex
%!TEX encoding =  UTF-8 Unicode
%!TEX root =  Phi+nEnglish.tex

%-------- newtheorem ----------

\theoremstyle{plain}
\newtheorem{theorem}{Theorem}[section]
\newtheorem{thdef}[theorem]{Theorem and definition}
\newtheorem{lemma}[theorem]{Lemma}
\newtheorem{corollary}[theorem]{Corollary}
\newtheorem{proposition}[theorem]{Proposition}
\newtheorem{propdef}[theorem]{Proposition and definition}
\newtheorem{plcc}[theorem]{Concrete local-global principle}
\newtheorem{fact}[theorem]{Fact}
\newtheorem{valsatz}[theorem]{\vst}

\newtheorem{theoremc}[theorem]{Theorem\etoz}
\newtheorem{lemmac}{Lemma\etoz}
\newtheorem{corollaryc}{Corollairy\etoz}
\newtheorem{propositionc}{Proposition\etoz}
\newtheorem{factc}{Fait\etoz}
\newtheorem*{Principleofcoveringbyquotients}{Principle of covering by quotients}

\theoremstyle{definition}
\newtheorem{conjecture}[theorem]{Conjecture}
\newtheorem{definition}[theorem]{Definition}
\newtheorem{definitions}[theorem]{Definitions}
\newtheorem{notation}[theorem]{Notation}
\newtheorem{definota}[theorem]{Definition and notation} 
\newtheorem{convention}[theorem]{Convention}
\newtheorem{problem}[theorem]{Problem}
\newtheorem{question}[theorem]{Question}

\theoremstyle{remark}
\newtheorem{remark}[theorem]{Remark}
\newtheorem{remarks}[theorem]{Remarks}
\newtheorem{comment}[theorem]{Comment}
\newtheorem{comments}[theorem]{Comments}
\newtheorem{example}[theorem]{Example}
\newtheorem{examples}[theorem]{Examples}

%% file: EnglishMacrosPhi-n.tex
%!TEX root =  Phi+nEnglish.tex

%:  Commentaires, remarques, exemples, problemes
\newcommand\comm{\rdb
\noi{\it Comment. }}

\newcommand\COM[1]{\rdb
\noi{\it Comment #1. }}

\newcommand\comms{\rdb
\noi{\it Comments. }}

\newcommand\Pb{\rdb
\noi{\bf Problem. }}

\newcommand \rem{\rdb
\noi{\sl Remark. }}

\newcommand \REM[1]{\rdb
\noi{\sl Remark#1. }}

\newcommand \rems{\rdb
\noi{\sl Remarks. }}

\newcommand \exl{\rdb
\noi{\bf Example. }}

\newcommand \EXL[1]{\rdb
\noi{\bf Example: #1. }}

\newcommand \exls{\rdb
\noi{\bf Examples. }}

\newcommand\gui[1]{``{#1}''}

\newcommand \thref[1] {Theorem~\ref{#1}}
\newcommand \paref[1] {page~\pageref{#1}}
\newcommand \pstfref[1] {Positivstellensatz formel~\ref{#1}}
\newcommand \pstref[1] {Positivstellensatz~\ref{#1}}

\newcommand \num {{n$^{\mathrm{ o}}$}}

\newcommand\subsubsec[1] {\subsubsection*{#1}}

% ----  ssi  etc  
\newcommand \hdr {induction hypothesis\xspace}
\newcommand \ssi {if and only if\xspace}
\newcommand \cnes {necessary and sufficient condition\xspace}
\newcommand \spdg {without loss of generality\xspace}
\newcommand \Propeq {T.F.A.E.\xspace}
\newcommand \propeq {t.f.a.e.\xspace}
\newcommand \disept {17$^{th}$ Hilbert's problem\xspace}

\newcommand \cad {\textit{i.e.}\xspace}

\newcommand \Vrai {\mathsf{True}}
\newcommand \Faux {\mathsf{False}}

%------- abr\'eviations math\'ematiques courantes ---

\newcommand \Amo {$\gA$-module\xspace}
\newcommand \Amos {$\gA$-modules\xspace}

\newcommand \Bmo {$\gB$-module\xspace}
\newcommand \Bmos {$\gB$-modules\xspace}

\newcommand \Zmo {$\gZ$-module\xspace}
\newcommand \Zmos {$\gZ$-modules\xspace}

\newcommand \ZZmo {$\ZZ$-module\xspace}
\newcommand \ZZmos {$\ZZ$-modules\xspace}

\newcommand \Ali {$\gA$-\ali}
\newcommand \Alis {$\gA$-\alis}

\newcommand \Alg {$\gA$-\alg}
\newcommand \Algs {$\gA$-\algs}

\newcommand \kev {$\gk$-vector space\xspace}
\newcommand \kevs {$\gk$-vector spaces\xspace}

\newcommand \Kev {$\gK$-vector space\xspace}
\newcommand \Kevs {$\gK$-vector spaces\xspace}

\newcommand \klg {$\gk$-\alg}
\newcommand \klgs {$\gk$-\algs}

\newcommand \Klg {$\gK$-\alg}
\newcommand \Klgs {$\gK$-\algs}

\newcommand \QQlg {$\QQ$-\alg}
\newcommand \QQlgs {$\QQ$-\algs}

\newcommand \ZZlg {$\QQ$-\alg}
\newcommand \ZZlgs {$\QQ$-\algs}

%: a
\newcommand \ac {algebraically closed\xspace}
\newcommand \alc {\agq closure\xspace}

\newcommand \adv {valuation domain\xspace}
\newcommand \advs {valuation domains\xspace}

\newcommand \arv {valuation ring\xspace}
\newcommand \arvs {valuation rings\xspace}

\newcommand \agq {algebraic\xspace}

\newcommand \alg {algebra\xspace}
\newcommand \algs {algebras\xspace}

\newcommand \agB {Boolean \alg}

\newcommand \algo{algorithm\xspace}
\newcommand \algos{algorithms\xspace}

\newcommand \algq{algorithmic\xspace}

\newcommand \ali {\lin map\xspace}
\newcommand \alis {\lin maps\xspace}

\newcommand \anar {\ari \ri}
\newcommand \anars {\ari \ris}
\newcommand \Anars {\Ari \ris}

\newcommand \ari{arith\-metic\xspace}

\newcommand \auto {automorphism\xspace}
\newcommand \autos {automorphisms\xspace}

%: c

\newcommand \cac {algebraically closed field\xspace}
\newcommand \cacs {algebraically closed fields\xspace}

\newcommand \cara{characteristic\xspace}  

\newcommand \carn{characterization\xspace}  
\newcommand \carns{characterizations\xspace}  

\newcommand \cdi{discrete field\xspace}  
\newcommand \cdis{discrete fields\xspace}  

\newcommand \cdf{fraction field\xspace}
\newcommand \cdfs{fraction fields\xspace}
 
\newcommand \cli{integral closure\xspace}  
\newcommand \clis{integral closures\xspace}  

\newcommand \codi {discrete ordered field\xspace}
\newcommand \codis {discrete ordered fields\xspace}

\newcommand \coe {coefficient\xspace}
\newcommand \coes {coefficients\xspace}

\newcommand \cof {\cov}

\newcommand \coh {coherent\xspace}

\newcommand \coli {linear combination\xspace}
\newcommand \colis {linear combinations\xspace}

\newcommand \com {comaximal\xspace}

\newcommand \coo {coordinate\xspace}
\newcommand \coos {coordinates\xspace}

\newcommand \cop {complementary\xspace}

\newcommand \corl {corollary\xspace}
\newcommand \corls {corollaries\xspace}

\newcommand \cosv {conservative\xspace}

\newcommand \cvd {valued discrete field\xspace}
\newcommand \cvds {valued discrete fields\xspace}

\newcommand \cvdsc {separably closed valued discrete field\xspace}
\newcommand \cvdscs {separably closed valued discrete fields\xspace}

\newcommand \cvdac {algebraicalle closed valued discrete field\xspace}
\newcommand \cvdacs {algebraicalle closed valued discrete fields\xspace}

\newcommand \cyct {cyclotomic\xspace}
\newcommand \cycts {cyclotomic\xspace}

%: d

\newcommand \dcd {residually discrete\xspace}

\newcommand \dcn {decomposition\xspace}
\newcommand \dcns {decompositions\xspace}

\newcommand \ddp {Pr\"ufer domain\xspace}
\newcommand \ddps {Pr\"ufer domains\xspace}

\newcommand \ddk {Krull dimension\xspace}

\newcommand \demo {proof\xspace}
\newcommand \dems {proofs\xspace}
\newcommand \demos {\dems}

\newcommand \dfn{definition\xspace}  
\newcommand \Dfn{Definition\xspace}  
\newcommand \Dfns{Definitions\xspace}  
\newcommand \dfns{definitions\xspace}  

\newcommand \dij{disjunctive\xspace}
\newcommand \wdij{weakly \dij}

\newcommand \discri{discriminant\xspace}
\newcommand \discris{discriminants\xspace}

\newcommand \dok {Dedekind domain\xspace}
\newcommand \doks {Dedekind domains\xspace}

\newcommand \dve {divisibility\xspace}

\newcommand \dvz {zerodivisor\xspace}
\newcommand \dvzs {zerodivisors\xspace}

%: e
\newcommand \eco{\com \elts}  

\newcommand \egmt{also\xspace} 

\newcommand \egt{equality\xspace} 
\newcommand \egts{equalities\xspace} 

\newcommand \elr{elementary\xspace}  

\newcommand \elt{element\xspace}  
\newcommand \elts{elements\xspace}  

\def \endo {endomorphism\xspace}
\def \endos {endomorphisms\xspace}

\newcommand \entrel {entailment relation\xspace}
\newcommand \entrels {entailment relations\xspace}

\newcommand \eqn  {equation\xspace}
\newcommand \eqns  {equations\xspace}

\newcommand \eqv  {equivalent\xspace}

\newcommand \eqvc  {equivalence\xspace}
\newcommand \eqvcs  {equivalences\xspace}

\newcommand \eseq{essentially equivalent\xspace} 
\newcommand \Eseq{Essentially equivalent\xspace} 

\newcommand \esid{essentially identical\xspace} 
\newcommand \Esid{Essentially identical\xspace} 

\newcommand \evc{vector space\xspace} 
\newcommand \evcs{vector spaces\xspace} 

%: f

\newcommand \fab {bounded \fcn}
\newcommand \fabs {bounded \fcns}

\newcommand \fac {total \fcn}

\newcommand \facile{\begin{proof}
Left to the reader.
\end{proof}}

\newcommand \fap {partial \fcn}
\newcommand \faps {partial \fcns}

\newcommand \fcn {factorization\xspace}
\newcommand \fcns {factorizations\xspace}

\newcommand \fdi{strongly discrete\xspace} 

%: g

\newcommand\gmq{geometric\xspace}

\newcommand\gne{generalised\xspace}

\newcommand\gnl{general\xspace}

\newcommand\gnlt{generally\xspace}

\newcommand\gnn{generalization\xspace}
\newcommand\gnns{generalizations\xspace}

\newcommand\gnq{generic\xspace}

\newcommand\grl{$\ell$-group\xspace}
\newcommand\grls{$\ell$-groups\xspace}

\newcommand \gtr{generator\xspace}  
\newcommand \gtrs{generators\xspace}  

%: h i

\newcommand \homo {homomorphism\xspace}
\newcommand \homos {homomorphisms\xspace}

\newcommand \id {ideal\xspace}
\newcommand \ids {ideals\xspace}

\newcommand \idd {de\-ter\-mi\-nantal \id}
\newcommand \idds {de\-ter\-mi\-nantal \ids}

\newcommand \idema {maximal \id}
\newcommand \idemas {maximal \ids}

\newcommand \idep {prime \id}
\newcommand \ideps {prime \ids}

\newcommand \idemi {minimal prime\xspace}
\newcommand \idemis {minimal primes\xspace}

\newcommand \idf {Fitting \id}
\newcommand \idfs {Fitting \ids}

\newcommand \idm {idempotent\xspace}
\newcommand \idms {idempotents\xspace}

\newcommand \idp {principal \id}
\newcommand \idps {principal \ids}

\newcommand \idtr {indeterminate\xspace}
\newcommand \idtrs {indeterminates\xspace}

\newcommand \ifr {fractional \id}
\newcommand \ifrs {fractional \ids}

\newcommand \inteq {intuitively \eqv}

\newcommand \ird {irreducible\xspace}
\newcommand \irds {irreducible\xspace}

\newcommand \itf {\tf \id}
\newcommand \itfs {\tf \ids}

\newcommand \iso {isomorphism\xspace}
\newcommand \isos {isomorphisms\xspace}

\newcommand \iv {invertible\xspace}
\newcommand \ivs {invertible\xspace}

%: l
\newcommand \lec {reader\xspace}

\newcommand \lgb {local global\xspace}

\newcommand \lin {linear\xspace}

\newcommand \lon {localization\xspace}
\newcommand \lons {localizations\xspace}

\newcommand \lop {\lot principal\xspace}

\newcommand \losd {\lot \sdz\xspace}

\def \lot {locally\xspace}

%: m
\newcommand \mlp {principal \lon matrix\xspace}
\newcommand \mlps {principal \lon matrices\xspace}

\newcommand \mnp {manipulation\xspace}
\newcommand \mnps {manipulations\xspace}
\newcommand \mnr {\elr \mnp}
\newcommand \mnrs {\elr \mnps}

\newcommand \mo {monoid\xspace}
\newcommand \mos {monoids\xspace}
\newcommand \moco {\com \mos}

\newcommand \mpf {\pf module\xspace}
\newcommand \mpfs {\pf modules\xspace}

\newcommand \mpn {\pn matrix\xspace}
\newcommand \mpns {\pn matrices\xspace}

\newcommand \mpr {\pro module\xspace}
\newcommand \mprs {\pro modules\xspace}

\newcommand \mprn {\prn matrix\xspace}
\newcommand \mprns {\prn matrices\xspace}

\newcommand \mptf {\ptf module\xspace}
\newcommand \mptfs {\ptf modules\xspace}

\newcommand \mrc {projective module of constant rank\xspace}
\newcommand \mrcs {projective modules of constant rank\xspace}

%: n

\newcommand \ncr{necessary\xspace}

\newcommand \ncrt{necessarily\xspace}

\newcommand \ndz {regular\xspace}

\newcommand \noe {Noetherian\xspace}
\newcommand \noco {\noe\coh}

\newcommand \nst {Nullstellensatz\xspace}
\newcommand \nsts {Nullstellens\"atze\xspace}

\newcommand \odz {Zariski open set\xspace}

\newcommand \oqc {\qc open set\xspace}
\newcommand \oqcs {\qc open sets\xspace}

\newcommand \ort {orthogonal\xspace}

%: p

\newcommand \pa {saturated pair\xspace}
\newcommand \pas {saturated pairs\xspace}

\newcommand \pb{problem\xspace}  
\newcommand \pbs{problems\xspace}

\newcommand \peq {purely equational\xspace}

\newcommand \pf {finitely presented\xspace}

\newcommand \plg {\lgb principle\xspace}
\newcommand \plgs {\lgb principles\xspace}

\newcommand \plga {abstract \plg}
\newcommand \plgas {abstract \plgs}

\newcommand \Plgc {Concrete \plg}
\newcommand \plgc {concrete \plg}
\newcommand \plgcs {concrete \plgs}

\newcommand \pn {presentation\xspace}
\newcommand \pns {presentations\xspace}

\newcommand \pol {polynomial\xspace}
\newcommand \pols {polynomials\xspace}

\newcommand \polcar {characteristic \pol}
\newcommand \polcars {characteristic \pols}

\newcommand \polmin {minimal \pol}
\newcommand \polmins {minimal \pols}

\newcommand \prc {rank constant \pro}

\newcommand \prmt {precisely\xspace}
\newcommand \Prmt {Precisely\xspace}

\newcommand \prn {projection\xspace}
\newcommand \prns {projections\xspace}

\newcommand \pro {projective\xspace}

\newcommand \proi {potential prime\xspace}
\newcommand \prois {potential primes\xspace}

\newcommand \proc {potential chain\xspace}
\newcommand \procs {potential chains\xspace}

\newcommand \proel {elementary \proc}
\newcommand \proels {elementary \procs}
\newcommand \proelo {\proel of length }
\newcommand \proelos {\proels of length }

\newcommand \prolo {\proc of length }
\newcommand \prolos {\procs of length }

\newcommand \prt {property\xspace}
\newcommand \prts {properties\xspace}

\newcommand \pst {Positivstellensatz\xspace}
\newcommand \psts {Positivstellens\"atze\xspace}

\newcommand \ptf {\tf \pro}

%: q

\newcommand \qc {quasi-compact\xspace}

\newcommand \qiri {pp-ring\xspace}
\newcommand \qiris {pp-rings\xspace}

\newcommand \ralg {Horn rule\xspace}
\newcommand \ralgs {Horn rules\xspace}

\newcommand \rcf {real closed field\xspace}
\newcommand \rcfs {real closed fields\xspace}

\newcommand \rdl {linear dependance relation\xspace}
\newcommand \rdls {linear dependance relations\xspace}

\newcommand \rdi {integral dependance relation\xspace}
\newcommand \rdis {integral dependance relations\xspace}

\newcommand \rdij {\dij rule\xspace}
\newcommand \rdijs {\dij rules\xspace}

\newcommand \rdv {valuative divisibility relation\xspace}
\newcommand \rdvs {valuative divisibility relations\xspace}

\newcommand \rdy {dynamical rule\xspace}
\newcommand \rdys {dynamical rules\xspace}

\newcommand \recu {induction\xspace}

\newcommand \red {direct rule\xspace}
\newcommand \reds {direct rules\xspace}

\newcommand \rex {existential rule\xspace}
\newcommand \rexs {existential rules\xspace}

\newcommand \ri {ring\xspace}
\newcommand \ris {rings\xspace}

%: s

\newcommand \sad {dynamical algebraic structure\xspace}
\newcommand \sads {dynamical algebraic structures\xspace}
\newcommand \SAD {Dynamical algebraic structure\xspace}
\newcommand \SADs {Dynamical algebraic structures\xspace}

\newcommand \salg {algebraic structure\xspace}
\newcommand \salgs {algebraic structures\xspace}

\newcommand \sdz {without \dvz}

\newcommand \sfio {fundamental system of orthogonal idempotents\xspace}

\newcommand \sgr {\gtr set\xspace}%{generator \sys}
\newcommand \sgrs {\gtr sets\xspace}%{generator \syss}

\newcommand \sli {\lin \sys}
\newcommand \slis {\lin \syss}

\newcommand \spb {separable\xspace}
\newcommand \spl {separable\xspace}

\newcommand \sps {spectral space\xspace}
\newcommand \spss {spectral spaces\xspace}

\newcommand \sys {system\xspace}
\newcommand \syss {systems\xspace}

%: t
\newcommand \talg {Horn theory\xspace}
\newcommand \talgs {Horn theories\xspace}

\newcommand \tco {coherent theory\xspace}
\newcommand \tcos {coherent theories\xspace}

\newcommand \twdij {\wdij theory\xspace}
\newcommand \twdijs {\wdij theories\xspace}

\newcommand \tdy {dynamical theory\xspace}
\newcommand \tdys {dynamical theories\xspace}

\newcommand \tel {regular theory\xspace}
\newcommand \tels {regular theories\xspace}

\newcommand \telri {cartesian theory\xspace}
\newcommand \telris {cartesian theories\xspace}

\newcommand \tf {finitely generated\xspace}

\newcommand \tfo {formal theory\xspace}
\newcommand \tfos {theory formelles\xspace}

\newcommand \tgm {\gmq theory\xspace}
\newcommand \tgms {\gmq theories\xspace}

\newcommand \Tho {Theorem\xspace}
\newcommand \tho {theorem\xspace}
\newcommand \thos {theorems\xspace}

\newcommand \tpe {purely equational theory\xspace}
\newcommand \tpes {purely equational theories\xspace}

\newcommand \uvl {universal\xspace}

\newcommand \trdi {distributive lattice\xspace}
\newcommand \trdis {distributive lattices\xspace}

\newcommand \vfn {verification\xspace}
\newcommand \vfns {verifications\xspace}

\newcommand \vst {Valuativstellensatz\xspace}
\newcommand \vsts {Valuativstellensätze\xspace}

\newcommand \zed {zero-dimensional\xspace}
\newcommand \zedr {zero-dimensional reduced\xspace}

%:  ------- maths constructives

\newcommand \cov {constructive\xspace}

\newcommand \coma {\cov \maths}
\newcommand \clama {classical \maths}

\renewcommand \cot {constructively\xspace}

\newcommand \mathe {mathematical\xspace}
\newcommand \maths {mathematics\xspace}

\newcommand \matn {mathematician\xspace}

\newcommand \pte {excluded middle principle\xspace}

\newcommand \prco {\cov proof\xspace}
\newcommand \prcos {constructive proofs\xspace}

\newcommand \tcg {compactness theorem\xspace}
\newcommand \Tcgi {The \tcg implies the following result. }

%% file: frenchPhi-n.tex
%!TEX root =  mainPhi-n.tex

\clearpage
\setcounter{page}{1} 
\setcounter{section}{0}
\setcounter{subsection}{0}
\setcounter{equation}{0}

\selectlanguage{french}
\def\frenchproofname{\textsl{Démonstration}}

\FrenchFootnotes

\input{FrenchTheoremsPhi-n.tex}
\input{FrenchMacrosPhi-n.tex}

\thickmuskip = 7mu plus 2mu

\pagestyle{headings}
\patchcmd{\sectionmark}{\MakeUppercase}{}{}{}

\stMF

\title{Polynômes cyclotomiques, sans utiliser les zéros des polynômes $Y^n-1$}

\author{Gema M. Diaz-Toca, Henri Lombardi, Claude Quitt\'e}

\date{\today}
\maketitle

\rdb
\label{beginfrench}

\begin{abstract} 
Cette note a pour but de construire un corps de racines \gui{intrinsèque} pour le \pol $X^n-1$ au dessus du corps $\QQ$ des rationnels, dans un esprit qui aurait sans doute plu~à Gauss, Kummer, Kronecker et Bishop.
Contrairement aux exposés usuels, notre construction n'utilise aucun corps de racines de $X^n-1$ qui serait donné avant la \demo d'irréductibilité du \pol \cyct.
\end{abstract}

%\noindent {\bf Mots clés}. 

%\smallskip \noindent {\bf MSC}. 

\section*{Introduction}

La définition la plus répandue du \textsl{$n$-ème \pol cyclotomique}  $\Phi_n(Y)$ utilise explicitement les racines dde l'unité dans $\CC$
\begin{equation}\label{fdef1}
\Phi_n(Y)=\prod\nolimits_\zeta(Y-\zeta),
\end{equation}
où $\zeta$ parcourt les racines primitives $n$-ème de l'unité, i.e.\ $\zeta^n=1$ et $\zeta^d\neq 1$ si $d\divi n$ et $d\neq n$.
 
Gauss a démontré que $\Phi_n(Y)\in\ZZY$ et aussi que $\Phi_p$ est \ird dans $\QQY$ quand $p$ est premier. Kronecker a démontré que $\Phi_n(Y)$ est \ird pour tout $n$.

Le \textsl{$n$-ème \pol cyclotomique} est défini comme $\Psi_n(Y)=\frac{Y^n-1}{\Phi_n(Y)}$.
 
Une autre \dfn bien connue des \pols cyclotomiques
est donné par \recu comme suit:  $\Phi_1(Y)= Y-1$ et 
\begin{equation}\label{fdef2}
\Phi_n(Y)=\frac{Y^n-1}{\prod\limits_{0<d<n,d\mid n}\Phi_d(Y)},\,\,\, n\geq 2.
\end{equation} 
Pour cette seconde \dfn, il est nécessaire de démontrer d'abord que \smash {$\prod\limits_{0<d<n,d\mid n}\Phi_d(Y)$} divise $Y^n-1$. Cela peut se déduire de la première \dfn. 

\smallskip Toutes les \demos  connues d'irréductibilité de 
 $\Phi_n(Y)$ dans $\QQ[Y]$ utilisent un corps de décomposition de $Y^n-1$. 

\smallskip Dans cette note nous proposons une nouvelle \dfn comme suit. Pour $n=1$, $\Psi_1(Y)=1$ et  $\Phi_1(Y)=Y-1$. Pour $n\geq 2$,
\begin{equation}\label{fdef3}
\Psi_n(Y)=\ppcm\limits_{d:0<d<n,d\mid n}(Y^d-1)  \; \text{ and }  \;\Phi_n(Y)=(Y^n-1)/\Psi_n(Y).
\end{equation}

Nous démontrons ensuite les principales \prts des \pols cyclotomiques, en particulier que
 $\Phi_n(Y)$ est \ird pour tout $n$. 

Nous obtenons de cette manière un corps de racines \textsl{intrinsèque} pour $Y^n-1$ sous la forme $\Qn:=\aqo{\QQY}{\Phi_n}$
d'une manière que Gauss, Kummer, Kronecker et Bishop auraient appréciée.

\smallskip
Cette note est écrite de manière pleinement \cov de sorte que les résultats des \demos peuvent être directement implémentées dans un système formel. En particulier les structures quotients sont toujours présentées via des \syss de congruences, dans l'esprit de Gauss, Kummer et Kronecker, renouvelé dans la version minimaliste \cov des ensembles par Bishop dans (\cite{fBi67}).   
 
\smallskip
Nous devons \egmt dire clairement que nous ne prétendons pas proposer un nouvel \algo particulièrement performant pour un calcul rapide des \pols cyclotomiques. À notre connaissance les \algos introduits par Arnold et Monagan dans  \cite{fArMo} restent les plus efficaces. En fait nous avons effectué des comparaisons en utilisant SAGE 10.4,  Mathematica 13 et Maple 2022. 
Dans SAGE, nous avons trouvé que notre \algo est de manière surprenante plus efficace que l'\algo par défaut lorsque l'entier $n$ est un produit conséquent de nombres premiers distincts. En outre nous avons observé que Maple donne les meilleurs temps d'exécution, et cela est probablement dû au fait que Maple utilise l'\algo  SPL de \cite{fArMo}.

\smallskip Pour des exposés classiques sur les \pols cyclotomiques nous recommandons particulièrement les travaux de \cite{fBro,fKro,fWei} et \cite{fTignol}.

\section{Préliminaires}

Nous rappelons d'abord quelques notations et résultats classiques
dans la théorie des \pols, que nous utiliserons dans la section \ref{fmain}. 

Nous notons $\PP$ l'ensemble des nombres premiers.

La \textsl{fonction \cara d'Euler} est la fonction $\phi\colon\NN\etl\to\NN\etl$ où $\phi(n)$ est le nombre de \gtrs d'un groupe cyclique à $n$ \elts.

Dans un anneau intègre à pgcd $\gZ$, nous notons $a\vi b$ pour $\pgcd(a,b)$ et
$a\vu b$ pour $\ppcm(a,b)$ ($a,b\neq 0$). 
Les \elts de $\gZ\etl$, lorsqu'on identifie deux \elts associés,
forment un treillis distributif, et $a(b\vi c)=(a b)\vi(a c)$. 

On rappelle maintenant le rapport entre la \dve dans $\ZX$ et celle dans $\QX$\footnote{La proposition classique \ref{fpropZXQX} se généralise en remplaçant $\Z$ par un anneau intègre à pgcd et $\QQ$ par son corps de fractions.}.
Nous notons $\rG(f)$ le pgcd des \coes de $f\in \ZX$. Le \pol $f$ est dit \textsl{G-primitif} si $\rG(f)=1$.

%p
%:     Proposition{fpropZXQX}
\begin{fproposition}[$\ZX$ est un anneau à pgcd] \label{fpropZXQX}~
\begin{enumerate}
\item  Soient  $f,g\in\QX$ des \pols unitaires. Si $fg\in\ZX$, alors $f$ \hbox{et $g\in\ZX$}. 
\item Soient $f,g\in\ZX$. Alors $f$ divise $g$ dans $\ZX$ \ssi $f$ divise $g$ dans $\QX$ et $\rG(f)$ divise $\rG(g)$ dans $\Z$.
\item L'anneau $\ZX$ est un anneau intègre à pgcd. Un \pol $f\in\ZX$ est \ird \ssi il est $\rG$-primitif et il est \ird dans $\QX$.
\item Deux \pols $\rG$-primitifs de $\ZX$ sont premiers entre eux (ils ont $1$ pour pgcd), \ssi ils sont premiers entre eux dans $\QX$.
\end{enumerate} 
\end{fproposition}
%----------- fin proposition ----------------------------- 

Rappelons qu'un corps est dit \textsl{discret} s'il possède un test à $0$, ce qui permet de traiter l'algèbre linéaire sur le corps de façon tout à fait explicite.

%f
%:     Fact{flemma1}
\begin{flemma} \label{flemma1}
Sur un corps discret $\gK$ arbitraire on a dans $\KX$  
\[(Y^m-1)\vi (Y^n-1)=Y^{ m\vi n}-1,
\]
pour $m,n\in\NN$.
\end{flemma}
%----------- fin fact ----------------------------------------- 
Ceci est basé sur l'\algo d'Euclide et sur le fait que si $r$ est le reste de la division de $n\geq 2$ par $m\geq 2$, alors $X^r-1$ est le reste de la division de $X^n-1$ par $X^m-1$.

%f
%:     Fact{flemma2}
\begin{flemma}[lemme de substitution] \label{flemma2}~\\ 
 Soit $\gK$ un corps et des \pols unitaires non constants $f_1,\dots,f_\ell,h\in\KX$. On note $g$ le pgcd (unitaire) des $f_i$ et $f$ leur ppcm (unitaire). On note $F_i(X)=f_i(h(X))$, $G(X)=g(h(X))$ et $F(X)=f(h(X))$. Alors $G$ est le pgcd des $F_i$ et $F$ est leur ppcm.
\end{flemma}
%----------- fin fact ----------------------------------------- 
%
\begin{proof} La substitution de $h(X)$ à $X$ est un homomorphisme de \Klgs, de $\KX$ dans lui-même: c'est un morphisme  d'évaluation. Il s'agit ici de voir que c'est \egmt un morphisme pour les lois de pgcd et ppcm. Cela se vérifierait pour tout morphisme entre anneaux de Bézout intègres, mais avec l'ambigüité due au fait que le pgcd et le pgcm ne sont bien définis qu'à association près. Dans le cas présent, on bénéficie du fait que tout \pol non nul est associé à un unique \pol unitaire.\\  
Vu l'associativité du pgcd et du ppcm, il suffit de traiter le cas où $\ell=2$. \\
Dans les relations de Bézout qui certifient que $g$ est le pgcd de $(f_1,f_2)$, on substitue $h(X)$~à~$X$ et on obtient les relations de Bézout qui certifient que $G$ est le pgcd de $(F_1,F_2)$.
\\
Pour le ppcm on substitue $h(X)$ à $X$ dans l'\egt $g(X)f(X)=f_1(X)f_2(X)$.
\end{proof}

%f
%:     Fact{flemma3}
\begin{flemma}[racines et \dcn] \label{flemma3} 
 Soit $\gA$ un anneau et un \pol unitaire $f\in \AY$ de degré $n$. Soient $x_1,\dots,x_m$  des zéros de $f$ dans $\gA$ avec les $x_i-x_j$ \ivs dans $\gA$ pour $i\neq j$. Alors on a:  
\begin{enumerate}
\item $f$ est divisible par $\prod_{i=1}^m(Y-x_i)$.
\item Si $m=n$, $f=\prod_{i=1}^n(Y-x_i)$.
\item Si $m>n$, l'anneau est trivial ($1=_{\gA}0$).
\end{enumerate}
\end{flemma}
%----------- fin fact ----------------------------------------- 
\facile

\section{La construction des \pols cyclotomiques}\label{fmain}

%d
%:     Definota{fdefiPolcyc}
\begin{fdefinota} \label{fdefiPolcyc}
Pour chaque entier $n>0$ on définit les \pols $\Phi_n(Y)$ et $\Psi_n(Y)\in\QQY$ comme suit.

\begin{itemize}
\item Pour $n=1$, $\Psi_1(Y)= 1$ et $\Phi_1(Y)=Y-1$.
\item Pour $n\geq 2$, $\Psi_n(Y)=\ppcm_{d:0<d<n,d\mid n}(Y^d-1)$ et $\Phi_n(Y)=(Y^n-1)/\Psi_n(Y)$.
\end{itemize}
On note aussi $\Gamma_n(Y)=(Y^n-1)/(Y-1)$, et donc $\Psi_n(Y)=(Y-1)\ppcm_{d:1<d<n,d\mid n}\Gamma_d(Y)$.
Enfin on note $\Qn:=\aqo{\Q[Y]}{\Phi_n(Y)}$. 
\end{fdefinota}
%----------- fin definition --------------------------------
Remarquez qu'il n'est pas affirmé pour le moment que $\Qn$ est un corps.

\smallskip Notez que puisque chaque $Y^d-1$ divise $Y^n-1$, il est est de même pour leur ppcm et le quotient $\Phi_n$ est un \elt bien défini dans $\QQY$. On a pris pour ppcm le \pol unitaire correspondant, donc le quotient est également unitaire.

\smallskip 
Par exemple nous avons
\begin{itemize}
\item $\Phi_2(Y)=Y+1$, $\Phi_3(Y)=Y^2+Y+1$,
\item $\Psi_4(Y)=Y^2-1$, $\Phi_4(Y)=Y^2+1$, 
\item $\Psi_8(Y)=Y^4-1$, $\Phi_8(Y)=Y^4+1$, 
\item pour $p\in\PP$, $\Psi_p=\Phi_1$ donc $\Phi_1\Phi_p=Y^p-1$, et $\Phi_p=\Gamma_p=\som_{r=0}^{p-1}Y^{r}$.
\end{itemize}

%l
%:     Lemma{flemPolcyc1}
\begin{flemma} \label{flemPolcyc1}
Pour chaque $n$, le \pol unitaire $\Phi_n(Y)$ est à \coes dans $\ZZ$. 
\end{flemma}
%----------- fin lemma ----------------------------------- 
%
\begin{proof}
Point \textsl{1} de la proposition \ref{fpropZXQX}. 
\end{proof}
%
%%l
%%:     Lemma{flemPolycyc01}
%\begin{flemma} \label{flemPolycyc01}
%Le \pol $\Phi_n$ est étranger à $\Psi_n$, et donc aux \pols  $Y^d-1$ pour $d$ divisant strictement $n$ ainsi qu'à leurs diviseurs. 
%\end{flemma}
%%----------- fin lemma ----------------------------------- 
%%
%\begin{proof}
%En effet, $Y^n-1$ est \spl et $Y^n-1=\Phi_n\Psi_n$.
%\end{proof}
%

%:     Lemma{flemPolcyc2}
\begin{flemma} \label{flemPolcyc2}
On a  
 $\Psi_n(Y)=\ppcm_{p:p\in\PP,p\mid n}(Y^{n/p}-1)$.
\end{flemma}
%----------- fin lemma ----------------------------------- 
%
\begin{proof}
Si $d$ divise strictement $n$ il divise  l'un des $n/p$ pour un $p\in\PP$ qui divise~$n$. Donc $Y^d-1$ divise $Y^{n/p}-1$ pour ce nombre premier $p.$ 
\end{proof}

Les deux \corls qui suivent sont surpassés par les résultats de la section suivante.

%c
%:     Corollary{fcoremPolcyc2}
\begin{fcorollary} \label{fcoremPolcyc2}
Si $p\in\PP$ divise $n$, on a $\Psi_{np}(Y)=\Psi_{n}(Y^p)$ et $\Phi_{np}(Y)=\Phi_{n}(Y^p)$\footnote{Pour les degrés cela correspond à $\phi(np)=p\phi(n)$.}. 
\end{fcorollary}
%--------- fin corollary ------------------------------- 
%
\begin{proof}
 Un nombre premier $q$ divise $np$ \ssi il divise $n$, donc  
 \[\Psi_{np}(Y)=\ppcm_{q:q\in\PP,q\mid n}(Y^{p {n/q}}-1)
 \]
qui est obtenu en substituant $Y^p$ à $Y$ dans dans $\Psi_{n}(Y)=\ppcm_{q:q\in\PP,q\mid n}(Y^{n/q}-1)$. 
Par le lemme de substitution on~a $\Psi_{np}(Y)=\Psi_{n}(Y^p)$. Par suite, $\Phi_{np}(Y)$ est obtenu en substituant~$Y^p$ à~$Y$ dans $(Y^n-1)/\Psi_{n}(Y)$. 
\end{proof}
%
%c
%:     Corollary{fcorcoremPolcyc2}
\begin{fcorollary} \label{fcorcoremPolcyc2}
Si tout facteur premier de $r$ divise $m$, on a $\Phi_{mr}(Y)=\Phi_{m}(Y^r)$. 
\end{fcorollary}
%--------- fin corollary ------------------------------- 
%
\begin{proof} Résulte du \corl précédent en utilisant la \dcn de $r$ en facteurs premiers.
\end{proof}

Le calcul de $\Phi_n$ est ainsi ramené au cas où $n$ est un entier sans facteur carré, produit de nombres premiers deux à deux distincts. Par exemple avec $360=12\times 30$, $\Phi_{360}(Y)=\Phi_{180}(Y^2)=\Phi_{90}(Y^4)=\Phi_{30}(Y^{12})$.

\section{Propriétés décisives}

%l
%:     Lemma{flemcrucial}
\begin{flemma} \label{flemcrucial} On est dans l'anneau $\ZY$. Soient $d,e\geq 2$.
\begin{enumerate}
\item Si $d \divi e$, alors $\Phi_d\vi(Y^e-1)=\Phi_d$.
\item Si $d \nedivi e$, alors $\Phi_d\vi(Y^e-1)=1$.
\item Si $d\neq e$, alors $\Phi_d\vi\Phi_e=1$.
\end{enumerate}
\textsl{1}. Si $d\divi e$, alors $\Phi_d$ divise $Y^d-1$ qui divise $Y^e-1$.

\smallskip \noindent  \textsl{2}. Si $d\nedivi e$, $d\vi e$ est un diviseur strict de $d$.\\
Or $\Phi_d\vi (X^e-1)$ divise $(X^d-1)\vi(X^e-1)=X^{d\vi e}-1$, qui divise $\Psi_d$ puisque $d\vi e$ est un diviseur strict de $d$. 
Donc $\Phi_d\vi (X^e-1)$, qui divise à la fois $\Phi_d$ et $\Psi_d$, est égal à $1$.

\smallskip \noindent \textsl{3}. L'un des deux ne divise pas l'autre. Par exemple
$d\nedivi e$. Alors, d'après le point \textsl{2} on~a~$\Phi_d\vi(Y^e-1)=1$. Comme $\Phi_e\divi Y^e-1$, on obtient $\Phi_d\vi\Phi_e=1$.

\end{flemma}
%----------- fin lemma ----------------------------------- 

%t
%:     Theorem{fMainth1}
\begin{ftheorem} \label{fMainth1}
Pour tout entier $n\geq 1$, Le \pol $Y^n-1$ se décompose dans $\Z[Y]$ sous la forme du produit de \pols deux à deux étrangers  
\[Y^n-1=\prod\nolimits_{d:d\mid n}\Phi_d(Y). \eqno{(1_n)}
\]
En particulier $\deg(\Phi_n)=\phi(n)$  (par \recu, car $n=\som\limits_{d:d\mid n}\phi(d)$).
\end{ftheorem}
%----------- fin theorem ----------------------------- 
%
\begin{proof} On fait une \demo par \recu. Initialisation avec $n=1$, ou si l'on préfère, $n$ un nombre premier. On suppose la \prt vraie pour tous les $d$ qui divisent strictement $n$. On remarque alors que $\Psi_n$ est le ppcm de tous les $\Phi_d$ pour les~$d$ qui divisent strictement $n$. En effet, par l'\hdr, chaque $Y^d-1$ est le produit des~$\Phi_h$ pour les $h$ qui divisent $d$ et c'est aussi leur ppcm car $Y^d-1$ est \spl.
On a donc
\[
Y^n-1=\Psi_n(Y)\,\Phi_n(Y)=
\bigg(\ppcm\limits_{d:d\mid n,1\leq d<n}(Y^d-1)\bigg)\,\Phi_n(Y)=
\bigg(\ppcm\limits_{d:d\mid n,1\leq d<n}\Phi_d(Y)\bigg)
\,\Phi_n(Y). 
\]
Dans le premier facteur du dernier produit, les $\Phi_d(Y)$ sont deux à deux étrangers d'après le lemme \ref{flemcrucial}, donc leur ppcm est égal à leur produit. Comme le \pol $Y^n-1$ est \spl, $\Phi_n$ est aussi étranger aux $\Phi_d(Y)$ pour $d$ diviseur strict de $n$.
\end{proof}

Sans conteste, l'un des résultats les plus fameux est l'\irdt de $\Phi_n$; voir différentes \demos dans \cite{fWei}. Nous en donnons une nouvelle.

%
%:     Theorem{fMainth2}
\begin{ftheorem} \label{fMainth2}
Pour tout entier $n\geq 1$, le \pol $\Phi_n$ est \ird dans $\QQY$.
%, et en notant $z\in\Qn$ l'\elt $Y$ vu modulo $\Phi_n$, on obtient dans  $\Qn[Y]$ l'\egt
%\[\Phi_n(Y)=\prod\nolimits_{r:1\leq r<n,r\vi n=1}(Y-z^r).
%\]  
\end{ftheorem}
%----------- fin theorem ----------------------------- 

%
\begin{proof}[Mise en route] Le résultat est clair pour $n=1$. \\
Soit $n\geq 2$ et $f$ un \pol unitaire de degré $\geq 1$ qui divise $\Phi_n$ dans $\QQY$. On sait que $f\in\ZZY$, on doit montrer qu'il est égal~à~$\Phi_n$. Il suffit de démontrer que $\deg(f)\geq\phi(n)$.
On considère la \QQlg $\QQ_f=\aqo\QQY{f}$. On note $y$ l'\elt $Y$ vu modulo $f$ de sorte que $\QQ_f=\QQy$. On démontre d'abord deux lemmes.
\end{proof}
%
%l
%:     Lemma{flem1Mainth2}
\begin{flemma} \label{flem1Mainth2}
On a $Y^n-1=\prod_{0\leq r<n}(Y-y^r)$ dans $\QQ_f[Y]$. 
\end{flemma}
%----------- fin lemma ----------------------------------- 
%
\begin{proof} Comme $f$ divise $Y^n-1$, on a $y^n=1$. Donc les $y^r$ sont des zéros de $Y^n-1$. Pour conclure, d'après le fait \ref{flemma3}, il suffit de voir que les $y^r-y^s$ pour $r>s$ sont \ivs dans $\QQy$. Écrivons $r=s+t$ et $y^r-y^s=y^s(y^t-1)$. 
Comme $n$ ne divise pas~$t$ et $f$ divise $\Phi_n$ on obtient $(Y^t-1)\vi f(Y)=1$ d'après le point \textsl{2} du lemme \ref{flemcrucial}.\\
La relation de Bézout dans $\QQ_f[Y]$ pour $(Y^t-1)\vi f(Y)=1$, spécialisée en $y$, montre \hbox{que $y^t-1$} est \iv dans $\QQ_f$.  
\end{proof}
%

%:     Lemma{flem2Mainth2}
\begin{flemma} \label{flem2Mainth2}
Si $r\vi n=1$ et $n\geq 2$, alors  $f(Y)$ divise $f(Y^r)$  dans $\QQY$ (ou, ce qui revient au même, dans $\ZZY$).% $\QQ_f[Y]$. 
\end{flemma}
%----------- fin lemma ----------------------------------- 
%
\begin{proof} Notez que si $f(Y)$ divise $f(Y^r)$, alors, par substitution 
$f(Y^k)$ divise $f(Y^{kr})$ pour tout $k$.
Il suffit donc de démontrer le résultat dans le cas où $r=p$, un nombre premier qui ne divise pas $n$.\\
Ecrivons $Y^n-1 = f(Y)g(Y)$ dans $\ZY$ donc $Y^{np}-1 = f(Y^p)g(Y^p)$. En utilisant $(Y^n-1) \mid (Y^{np}-1)$, il vient :
$f(Y) \mid f(Y^p)g(Y^p)$.
Passons dans $\Fp[Y]$. Puisque $p\nedivi n$, on a $\ov f(Y)\vi \ov g(Y) = \ov1$ donc $\ov f(Y)\vi \ov g(Y)^p = \ov1$. Comme $\ov g(Y)^p = \ov g(Y^p)$, on obtient $\ov f(Y) \vi \ov g(Y^p) = \ov 1$, et à fortiori $f(Y) \vi g(Y^p) = 1$ dans $\ZZY$. Cela implique $f(Y) \mid f(Y^p)$.
\end{proof}
\begin{proof}[Fin de la \demo du \thref{fMainth2}] D'après le lemme précédent, en spécialisant $Y$ en $y$ on obtient que  $f(y^r)=0$ pour $r<n$ étranger à $n$,
or les $y^r-y^{r'}$ sont \ivs dans~$\QQ_f$ (déjà démontré).  
Donc $f(Y)$ est multiple de $\prod_{r:1\leq r<n,r\vi n=1}(Y-y^r)$ d'après le fait \ref{flemma3}. Ainsi $\deg(f)\geq \phi(n)$.
En conclusion, $f=\Phi_n$, $\Qn=\QQ_f$ et
\[\Phi_n(Y)=f(Y)=\prod\nolimits_{r:1\leq r<n,r\vi n=1}(Y-y^r).
\] 
\end{proof}

\noindent \textsl{Remarque}. Précisions.
\begin{itemize}
\item La même \demo que celle du lemme \ref{flem1Mainth2} donne $Y^n-1=\prod_{0\leq r<n}(Y-z^r)$ dans $\Qn[Y]$.
\item La  \demo du lemme \ref{flem2Mainth2} dit en fait plus \prmt que si un \pol unitaire $F\in \ZZY$ divise $Y^n-1$ et si $r\vi n = 1$, alors $F(Y)$ divise $F(Y^r)$  dans $\QQY$. \eoe
\end{itemize}

\smallskip Le corolaire qui suit établit l'équivalence entre notre \dfn et la \dfn classique des \pols \cycts.

%c
%:     Corollary{fcorMainths0}
\begin{fcorollary} \label{fcorMainths0}
Soit $\gK$ une extension de $\QQ$ qui contient une racine primitive de l'unité d'ordre $n$, que nous notons $\xi$. 
Alors le sous-corps $\QQ[\xi]$ est isomorphe à $\Qn$.
Un tel isomorphisme $\varphi\colon \QQ[\xi] \to \Qn$ est déterminé par l'image de $\xi$ qui est n'importe quelle racine primitive $n$-ème de l'unité dans $\Qn$. 
\end{fcorollary}
%--------- fin corollary ------------------------------- 
%
\begin{proof} En tant que racine primitive $n$-ème de l'unité, $\xi$ annule $X^n-1$ mais n'annule aucun $X^d-1$ si $d$ divise strictement $n$.
En particulier $\Phi_d(\xi)$ est \iv dans~$\gK$. L'\egt $\xi^n-1=\prod_{d \mid n}\Phi_d(\xi)$ montre alors que $\Phi_n(\xi)=0$. Comme $\Phi_n(Y)$ est \ird sur~$\QQ$, c'est le \polmin de $\xi$ au dessus de $\QQ$. 
\end{proof}
%

%c
%:     Corollary{fcorMainths}
\begin{fcorollary} \label{fcorMainths}
Si $n$ est pair, le groupe des racines de l'unité de $\Qn$ est d'ordre $n$, engendré par $x_n$. Si $n$ est impair, le groupe des racines de l'unité de $\Qn$ est d'ordre $2n$, engendré par $-x_n$. 
\end{fcorollary}
%--------- fin corollary ------------------------------- 
%
\begin{proof}
Supposons $n$ pair. Si le groupe des racines de l'unité était strictement plus grand, il contiendrait un $y$ d'ordre $nk$ pour $k>1$. Cet $y$ engendrerait un sous-corps isomorphe à $\Q_{nk}$ (\corl \ref{fcorMainths0}) de dimension $\phi(nk)$. Or puisque $n$ est pair $\phi(nk)>\phi(n)$.\\
Si $n$ est impair, $-x_n$ est d'ordre $2n$ dans $\Qn$, donc $\Qn=\QQ_{2n}$ et le groupe des racines de l'unité ne peut pas être d'ordre $>2n$ d'après le cas $n$ pair. 
\end{proof}

%

%:     corollary{fcorMainths3}
\begin{fcorollary} \label{fcorMainths3}
Soit $n$ pair. Le corps $\Qn$ contient un sous-corps isomorphe à $\Q_d$ \ssi $d$ divise $n$.  
\end{fcorollary}
%----------- fin cor ----------------------------------- 
%:     corollary{fcorMainths4}
\begin{fcorollary} \label{fcorMainths4}
Si $n$ est impair $\geq 3$, $\Phi_{2n}(Y)= \Phi_{n}(-Y)$. 
\end{fcorollary}
%----------- fin cor ----------------------------------- 

\medskip \noindent 
{\bf Conclusion.} L'approche \gui{intrinsèque} que nous proposons nous semble plus élémentaire que l'approche usuelle plus savante qui utilise l'existence d'un corps de racines pour le polynôme $X^n-1$. Cependant, bien qu'il soit assez facile de démontrer que tout groupe de racines de l'unité dans un corps est nécessairement cyclique, nous regrettons de ne pas avoir réussi à exploiter ce fait pour simplifier de manière radicale la démonstration de l'irréductibilité de $\Phi_n$. 

\bibliographystyle{plainnat-fr}
%\bibliography{PhinbibF.bib}

\end{document}

%% file: FrenchTheoremsPhi-n.tex
%!TEX encoding =  UTF-8 Unicode

%-------- newtheorem ----------

%\theoremstyle{theorem}
\newtheorem{ftheorem}{Théorème}[section]
\newtheorem{flemma}[ftheorem]{Lemme}%[section]
\newtheorem{fcorollary}[ftheorem]{Corolaire}%[section]
\newtheorem{fproposition}[ftheorem]{Proposition}%[section]

\theoremstyle{definition}
\newtheorem{fdefinition}[ftheorem]{Définition}%[section]
\newtheorem{fdefinota}[ftheorem]{Définition et notation}

%\theoremstyle{remark}

%% file: FrenchMacrosPhi-n.tex
%!TEX encoding =  UTF-8 Unicode
%!TEX root =  Phi+n.tex
\newcommand{\vou}{\MA{\tsbf{ ou }}}
\newcommand{\Vou}{\MA{\tsbf{OU}}}
\newcommand \EXists[1] {\tsbf{Introduire }{#1}\tsbf{ tel que }\,}
\newcommand \vet {{\tsbf{,}}\,}
\newcommand \Atcl {\mathrm{Atcl}}

%%%%%%%%%%%%%%%%%%%%%%%%%%%%%%%%%%%%%%%%%

%:  lecteur, lectrice, masculin feminin

\newcounter{MF}
\newcommand\stMF{\stepcounter{MF}}

\newcommand{\lec}{\stMF\ifodd\value{MF}lecteur\xspace\else 
lectrice\xspace\fi}

\newcommand{\lecs}{\stMF\ifodd\value{MF}lecteurs\xspace\else 
lectrices\xspace\fi}

\newcommand{\alec}{\stMF\ifodd\value{MF}au lecteur\xspace\else%
à la lectrice\xspace\fi}

\newcommand{\dlec}{\stMF\ifodd\value{MF}du lecteur\xspace\else%
de la lectrice\xspace\fi}

\newcommand{\llec}{\stMF\ifodd\value{MF}le lecteur\xspace\else la lectrice\xspace\fi}

\newcommand{\Llec}{\stMF\ifodd\value{MF}Le lecteur\xspace\else La lectrice\xspace\fi}

% les suivants ne changent pas le genre
\newcommand{\lui}{\ifodd\value{MF}lui\xspace\else
elle\xspace\fi}

\newcommand{\celui}{\ifodd\value{MF}celui\xspace\else
celle\xspace\fi}

\newcommand{\ceux}{\ifodd\value{MF}ceux\xspace\else
celles\xspace\fi}

\newcommand{\er}{\ifodd\value{MF}er\xspace\else
ère\xspace\fi}

\newcommand{\eux}{\ifodd\value{MF}eux\xspace\else
elles\xspace\fi}

\newcommand{\eUx}{\ifodd\value{MF}eux\xspace\else
euse\xspace\fi}

\newcommand{\eUX}{\ifodd\value{MF}eux\xspace\else
euses\xspace\fi}

\newcommand{\leux}{\ifodd\value{MF}leux\xspace\else
leuse\xspace\fi}

\newcommand{\il}{\ifodd\value{MF}il\xspace\else
elle\xspace\fi}

\newcommand{\ien}{\ifodd\value{MF}ien\xspace\else
ienne\xspace\fi}

\newcommand{\e}{\ifodd\value{MF}\xspace \else e\xspace\fi}

\newcommand{\n}{\ifodd\value{MF}n\xspace\else nne\xspace\fi}

%  \la* fait le ou la sans changer le genre,
%  \la fait le ou la en changeant de genre
\makeatletter
\newcommand{\la}{\@ifstar{\ifodd\value{MF}le\else
la\fi}{\stMF\ifodd\value{MF}le\else la\fi}}
\makeatother
%%%%%%%%%%%%%%%%%%%%%%%%%%%%%%%%%%%%%%%%%%%%%%%%%%%%%%%%%%%%%%%%%%%%

\newcommand \rem{\rdb
\noi{\it Remarque. }}

\newcommand \REM[1]{\rdb
\noi{\it Remarque#1. }}

\newcommand \rems{\rdb
\noi{\it Remarques. }}

\newcommand \exl{\rdb
\noi{\bf Exemple. }}

\newcommand \EXL[1]{\rdb
\noi{\bf Exemple: #1. }}

\newcommand \exls{\rdb
\noi{\bf Exemples. }}

\newcommand \thref[1] {théorème~\ref{#1}}
\newcommand \paref[1] {page~\pageref{#1}}
\newcommand \pstfref[1] {Posi\-tiv\-stel\-lensatz formel~\ref{#1}}
\newcommand \pstref[1] {Posi\-tiv\-stel\-lensatz~\ref{#1}}

\newcommand\oge{\leavevmode\raise.3ex\hbox{$\scriptscriptstyle\langle\!\langle\,$}}
\newcommand\feg{\leavevmode\raise.3ex\hbox{$\scriptscriptstyle\,\rangle\!\rangle$}}

\newcommand\gui[1]{\oge{#1}\feg}

\newcommand \facile{\begin{proof}
%La démonstration est laissée \alec.
Laissée \alec.
\end{proof}
}

\newcommand \num {{n$^{\mathrm{ o}}$}}
%\newcommand \num {{n\o}}

%:  Commentaires, remarques, exemples, problemes
\newcommand\comm{\rdb
\noi{\it Commentaire. }}

\newcommand\COM[1]{\rdb
\noi{\it Commentaire #1. }}

\newcommand\comms{\rdb
\noi{\it Commentaires. }}

\newcommand\Subsubsection[1]{%goodbreak
\rdb\addcontentsline{toc}{subsubsection}{#1} \subsubsection*{#1}}

\renewcommand\paragraph[1]{

%\rdb\addcontentsline{toc}{subsubsection}{#1} 
\medskip \noindent $\bullet$ \textbf{#1}}

\newcommand\Pb{\rdb
\noi{\bf Problème. }}

\newcommand\eoq{\hbox{}\nobreak
\vrule width 1.4mm height 1.4mm depth 0mm}

%:   cad hdr spdg propeq ...
\newcommand \Cad {C'est-à-dire\xspace}
\newcommand \recu {récur\-rence\xspace}
\newcommand \hdr {hypo\-thèse de \recu}
\newcommand \cad {c'est-à-dire\xspace}
\newcommand \cade {c'est-à-dire en\-co\-re\xspace}
\newcommand \ssi {si, et seu\-lement si, }
\newcommand \ssiz {si, et seu\-lement si,~}
\newcommand \cnes {con\-di\-tion néces\-sai\-re et suf\-fi\-san\-te\xspace}
\newcommand \spdg {sans per\-te de géné\-ra\-lité\xspace}
\newcommand \Spdg {Sans per\-te de géné\-ra\-lité\xspace}

\newcommand \Propeq {Les pro\-pri\-é\-tés sui\-van\-tes sont 
équi\-va\-len\-tes.}
\newcommand \propeq {les pro\-pri\-é\-tés sui\-van\-tes sont 
équi\-va\-len\-tes.}

%:  Kev Alg etc...
\newcommand \Kev {$\gK$-\evc}
\newcommand \Kbev {$\gKb$-\evc}
\newcommand \Kevs {$\gK$-\evcs}

\newcommand \Lev {$\gL$-\evc}
\newcommand \Levs {$\gL$-\evcs}

\newcommand \Qev {$\QQ$-\evc}
\newcommand \Qevs {$\QQ$-\evcs}

\newcommand \QQev {$\QQ$-\evc}
\newcommand \QQevs {$\QQ$-\evcs}

\newcommand \kev {$\gk$-\evc}
\newcommand \kevs {$\gk$-\evcs}

\newcommand \lev {$\gl$-\evc}
\newcommand \levs {$\gl$-\evcs}

\newcommand \Alg {$\gA$-\alg}
\newcommand \Algs {$\gA$-\algs}

\newcommand \Blg {$\gB$-\alg}
\newcommand \Blgs {$\gB$-\algs}

\newcommand \Clg {$\gC$-\alg}
\newcommand \Clgs {$\gC$-\algs}

\newcommand \klg {$\gk$-\alg}
\newcommand \klgs {$\gk$-\algs}

\newcommand \llg {$\gl$-\alg}
\newcommand \llgs {$\gl$-\algs}

\newcommand \Klg {$\gK$-\alg}
\newcommand \Klgs {$\gK$-\algs}

\newcommand \Llg {$\gL$-\alg}
\newcommand \Llgs {$\gL$-\algs}

\newcommand \QQlg {$\QQ$-\alg}
\newcommand \QQlgs {$\QQ$-\algs}

\newcommand \Rlg {$\gR$-\alg}
\newcommand \Rlgs {$\gR$-\algs}

\newcommand \RRlg {$\RR$-\alg}
\newcommand \RRlgs {$\RR$-\algs}

\newcommand \ZZlg {$\ZZ$-\alg}
\newcommand \ZZlgs {$\ZZ$-\algs}

\newcommand \Amo {$\gA$-mo\-du\-le\xspace}
\newcommand \Amos {$\gA$-mo\-du\-les\xspace}

\newcommand \Bmo {$\gB$-mo\-du\-le\xspace}
\newcommand \Bmos {$\gB$-mo\-du\-les\xspace}

\newcommand \Cmo {$\gC$-mo\-du\-le\xspace}
\newcommand \Cmos {$\gC$-mo\-du\-les\xspace}

\newcommand \kmo {$\gk$-mo\-du\-le\xspace}
\newcommand \kmos {$\gk$-mo\-du\-les\xspace}

\newcommand \Kmo {$\gK$-mo\-du\-le\xspace}
\newcommand \Kmos {$\gK$-mo\-du\-les\xspace}

\newcommand \Lmo {$\gL$-mo\-du\-le\xspace}
\newcommand \Lmos {$\gL$-mo\-du\-les\xspace}

\newcommand \Vmo {$\gV$-mo\-du\-le\xspace}
\newcommand \Vmos {$\gV$-mo\-du\-les\xspace}

\newcommand \Ali {appli\-ca\-tion $\gA$-\lin}
\newcommand \Alis {appli\-ca\-tions $\gA$-\lins}

\newcommand \Kli {appli\-ca\-tion $\gK$-\lin}
\newcommand \Klis {appli\-ca\-tions $\gK$-\lins}

\newcommand \Bli {appli\-ca\-tion $\gB$-\lin}
\newcommand \Blis {appli\-ca\-tions $\gB$-\lins}

\newcommand \Cli {appli\-ca\-tion $\gC$-\lin}
\newcommand \Clis {appli\-ca\-tions $\gC$-\lins}

%:  a
\newcommand \ac{algé\-bri\-quement clos\xspace}  

\newcommand \acl {an\-neau \icl}
\newcommand \acls {an\-neaux \icl}

\newcommand \adp {an\-neau de Pr\"u\-fer\xspace}
\newcommand \adps {an\-neaux de Pr\"u\-fer\xspace}

\newcommand \adpc {\adp \coh}
\newcommand \adpcs {\adps \cohs}

\newcommand \adu {\alg de décom\-po\-sition univer\-selle\xspace}
\newcommand \adus {\algs de décom\-po\-sition univer\-selle\xspace}

\newcommand \adv {anneau de valuation\xspace}
\newcommand \advs {anneaux de valuation\xspace}

\newcommand \advd {anneau de valuation discrète\xspace}
\newcommand \advds {anneaux de valuation discrète\xspace}

\newcommand \advl {anneau \dvla} %anneaux à difiseurs
\newcommand \advls {anneaux \dvlas} 

\newcommand \Afr {Anneau \frl}
\newcommand \Afrs {Anneaux \frls}
\newcommand \afr {anneau \frl}
\newcommand \afrs {anneaux \frls}

\newcommand \aFr {\hyperref[theorieAfr]{anneau \frl}}
\newcommand \aFrs {\hyperref[theorieAfr]{anneau \frls}}

\newcommand \afrr {\afr réduit\xspace}
\newcommand \afrrs {\afrs réduits\xspace}
\newcommand \Afrrs {\Afrs réduits\xspace}

\newcommand \afrvr {\afr avec \ravs}
\newcommand \aFrvr {\hyperref[theorieAfrrv]{\afrvr}}
\newcommand \afrvrs {\afrs avec \ravs}

\newcommand \aftr {anneau réticulé \ftm réel\xspace}
\newcommand \aftrs {anneaux réticulés \ftm réels\xspace}

\newcommand \aG {\alg galoisienne\xspace}
\newcommand \aGs {\algs galoisiennes\xspace}

\newcommand \agB {\alg de Boole\xspace}
\newcommand \agBs {\algs de Boole\xspace}

\newcommand \agH {\alg de Heyting\xspace}
\newcommand \agHs {\algs de Heyting\xspace}

\newcommand \agq{algébrique\xspace}
\newcommand \agqs{algébriques\xspace}

\newcommand \agqt{algébriuement\xspace}

\newcommand \aKr {anneau de Krull\xspace}
\newcommand \aKrs {anneaux de Krull\xspace}

\newcommand \alg {algè\-bre\xspace}
\newcommand \algs {algè\-bres\xspace}

\newcommand \algo{algo\-rithme\xspace}
\newcommand \algos{algo\-rithmes\xspace}

\newcommand \algq{al\-go\-rith\-mi\-que\xspace}
\newcommand \algqs{al\-go\-rith\-mi\-ques\xspace}

\newcommand \ali {appli\-ca\-tion \lin}
\newcommand \alis {appli\-ca\-tions \lins}

\newcommand \alo {an\-neau lo\-cal\xspace}
\newcommand \alos {an\-neaux lo\-caux\xspace}

\newcommand \algb {an\-neau \lgb}
\newcommand \algbs {an\-neaux \lgbs}

\newcommand \alrd {\alo \dcd}
\newcommand \alrds {\alos \dcds}

\newcommand \anar {anneau \ari}
\newcommand \anars {anneaux \aris}

\newcommand \anor {an\-neau nor\-mal\xspace}
\newcommand \anors {an\-neaux nor\-maux\xspace}

\newcommand \apf {\alg \pf}
\newcommand \apfs {\algs \pf}

\newcommand \apG {\alg pré\-galoisienne\xspace}
\newcommand \apGs {\algs pré\-galoisiennes\xspace}

\newcommand \arch {archimédien\xspace}
\newcommand \arche {archimédienne\xspace}
\newcommand \archs {archimédiens\xspace}
\newcommand \arches {archimédiennes\xspace}

\newcommand \arc {anneau réel clos\xspace}
\newcommand \aRc {\hyperref[theorieArc]{\arc}}
\newcommand \arcs {anneaux réels clos\xspace}

\newcommand \ari{arith\-mé\-tique\xspace}  
\newcommand \aris{arith\-mé\-tiques\xspace}  

\newcommand \Asr {Anneau \str}
\newcommand \Asrs {Anneaux \strs}
\newcommand \asr {anneau \str}
\newcommand \asrs {anneaux \strs}

\newcommand \asrvr {\asr avec \ravs}
\newcommand \asrvrs {\asrs avec \ravs}

\newcommand \atf {\alg \tf}
\newcommand \atfs {\algs \tf}

\newcommand \auto {auto\-mor\-phisme\xspace}
\newcommand \autos {auto\-mor\-phismes\xspace}

%:  b 

\newcommand \bdg {base de Gr\"obner\xspace}
\newcommand \bdgs {bases de Gr\"obner\xspace}

\newcommand \bdp {base de \dcn partielle\xspace}
\newcommand \bdps {bases de \dcn partielle\xspace}

\newcommand \bdf {base de \fap\xspace}

\newcommand \Bif {Borne infé\-rieure\xspace} %
\newcommand \bif {borne infé\-rieure\xspace} %
\newcommand \bifs {bornes infé\-rieures\xspace} %

\newcommand \bsp {borne supé\-rieure\xspace} %
\newcommand \bsps {borne supé\-rieures\xspace} %

%:  c

\newcommand \cac{corps \ac}  

\newcommand \calf{calcul formel\xspace}  

\newcommand \cara{carac\-té\-ris\-tique\xspace}  
\newcommand \caras{carac\-té\-ris\-tiques\xspace}  

\newcommand \carn{carac\-té\-ri\-sation\xspace}  
\newcommand \carns{carac\-té\-ri\-sations\xspace}  

\newcommand \carar{carac\-té\-riser\xspace}

\newcommand \carf{de carac\-tère fini\xspace}  

\newcommand \cdi{corps discret\xspace}
\newcommand \cdis{corps discrets\xspace}
  
\newcommand \cdv{changement de variables\xspace}  
\newcommand \cdvs{changements de variables\xspace}  

\newcommand \cdva{corps valué discret\xspace}
\newcommand \cdvas{corps valué discrets\xspace}

\newcommand \cla {clôture \agq}
\newcommand \clas {clôtures \agqs}

\newcommand \cli {clôture intégrale\xspace}
\newcommand \clis {clôtures intégrales\xspace}

\newcommand \codi {corps ordonné discret\xspace}
\newcommand \codis {corps ordonnés discrets\xspace}

\newcommand \coe {coef\-fi\-cient\xspace}
\newcommand \coes {coef\-fi\-cients\xspace}

\newcommand \coh {co\-hé\-rent\xspace}
\newcommand \cohs {co\-hé\-rents\xspace}

\newcommand \cohc {co\-hé\-rence\xspace}

\newcommand \coli {combi\-nai\-son \lin}
\newcommand \colis {combi\-nai\-sons \lins}

\newcommand \com {co\-ma\-xi\-maux\xspace}
\newcommand \come {co\-ma\-xi\-males\xspace}

\newcommand \coo {coor\-donnée\xspace}
\newcommand \coos {coor\-données\xspace}

\newcommand \cop {complé\-men\-taire\xspace}
\newcommand \cops {complé\-men\-taires\xspace}

\newcommand \corl {coro\-laire\xspace}
\newcommand \corls {coro\-laires\xspace}

\newcommand \cosv {conser\-vative\xspace}
\newcommand \cosvs {conser\-vatives\xspace}

\newcommand \cOsv {\hyperref[defithconserv]{conser\-vative\xspace}}
\newcommand \cOsvs {\hyperref[defithconserv]{conser\-vatives\xspace}}

\newcommand \covr {corps ordonné avec \ravs}
\newcommand \covrs {corps ordonnés avec \ravs}

\newcommand \cpb {compa\-tible\xspace} 
\newcommand \cpbs {compa\-tibles\xspace} 

\newcommand \cpbt {compa\-tibi\-lité\xspace} 
\newcommand \cpbtz {compa\-tibi\-lité} 

\newcommand \crc {corps réel clos\xspace}
\newcommand \crcs {corps réels clos\xspace}

\newcommand \crcd {corps réel clos discret\xspace}
\newcommand \crcds {corps réels clos discrets\xspace}

\newcommand \cyct {cyclo\-to\-mique\xspace}
\newcommand \cycts {cyclo\-to\-miques\xspace}

%:  d 

\newcommand \dcd {rési\-duel\-lement dis\-cret\xspace}
\newcommand \dcds {rési\-duel\-lement dis\-crets\xspace}

\newcommand \dcn {décom\-po\-sition\xspace}
\newcommand \dcns {décom\-po\-sitions\xspace}

\newcommand \dcnb {\dcn bornée\xspace}

\newcommand \dcnc {\dcn complète\xspace}

\newcommand \dcnp {\dcn partielle\xspace}

\newcommand \dcp {décom\-posa\-ble\xspace}
\newcommand \dcps {décom\-posa\-bles\xspace}

\newcommand \ddk {dimension de~Krull\xspace}
\newcommand \ddi {de dimension infé\-rieure ou égale à~}

\newcommand \dimm {description immédiate\xspace}

\newcommand \ddp {domaine de Pr\"u\-fer\xspace}
\newcommand \ddps {domaines de Pr\"u\-fer\xspace}

\newcommand \Demo{Démonstration\xspace}     
\newcommand \Demos{Démonstrations\xspace}     

\newcommand \demo{démon\-stra\-tion\xspace}     
\newcommand \demos{démon\-stra\-tions\xspace}     

\newcommand \dems{démons\-tra\-tions\xspace}

\newcommand \deno{déno\-mi\-nateur\xspace}     
\newcommand \denos{déno\-mi\-nateurs\xspace}   

\newcommand \deter {déter\-mi\-nant\xspace}  
\newcommand \deters {déter\-mi\-nants\xspace}  
  
\newcommand \Dfn{Défi\-nition\xspace}  
\newcommand \Dfns{Défi\-nitions\xspace}  
\newcommand \dfn{défi\-nition\xspace}  
\newcommand \dfns{défi\-nitions\xspace}  

\newcommand \dftr {droite réticulée \ftm réelle\xspace}
\newcommand \dftrs {droites réticulées \ftm réelles\xspace}
  
\newcommand \dil{diffé\-rentiel\xspace}  
\newcommand \dils{diffé\-rentiels\xspace}  
\newcommand \dile{diffé\-ren\-tielle\xspace}  
\newcommand \diles{diffé\-ren\-tielles\xspace}  

\newcommand \dip{diviseur principal\xspace}
\newcommand \dips{diviseurs principaux\xspace}

\newcommand \discri{discri\-minant\xspace}  
\newcommand \discris{discri\-minants\xspace}  

\newcommand \divle {dimension divisorielle\xspace}

\newcommand \dit{distri\-bu\-ti\-vité\xspace}

\newcommand \dlg{d'élar\-gis\-sement\xspace}  

\newcommand \dok {domaine de Dedekind\xspace}
\newcommand \doks {domaines de Dedekind\xspace}

\newcommand \dvla {à diviseurs\xspace}% {diviso\-riel\xspace} %anneaux
\newcommand \dvlas {à diviseurs\xspace}% %anneaux

\newcommand \dvld {\dvlt décom\-posé\xspace} %
\newcommand \dvlds {\dvlt décom\-posés\xspace} %

\newcommand \dvlg {diviso\-riel\xspace} %groupes
\newcommand \dvlgs {diviso\-riels\xspace} %groupes

\newcommand \dvli {\dvlt inver\-sible\xspace} %ideal
\newcommand \dvlis {\dvlt inver\-sibles\xspace} %ideal

\newcommand \dvlt {diviso\-riel\-lement\xspace} %

\newcommand \dvz {di\-viseur de zéro\xspace}
\newcommand \dvzs {di\-viseurs de zéro\xspace}

\newcommand \dve {divi\-si\-bi\-lité\xspace}

\newcommand \dvee {à \dve explicite\xspace}

\newcommand \dvr {diviseur\xspace}
\newcommand \dvrs {diviseurs\xspace}

%:  e

\newcommand \Eds {Exten\-sion des sca\-laires\xspace}
\newcommand \edss {exten\-sions des sca\-laires\xspace}
\newcommand \eds {exten\-sion des sca\-laires\xspace}

\newcommand \eco {\elts \com}

\newcommand \egmt {éga\-lement\xspace}

\newcommand \egt {éga\-li\-té\xspace}
\newcommand \egts {éga\-li\-tés\xspace}

\newcommand \eli{élimi\-nation\xspace}  

\newcommand \elr{élé\-men\-taire\xspace}  
\newcommand \elrs{élé\-men\-taires\xspace}  

\newcommand \elrt{élé\-men\-tai\-rement\xspace}  

\newcommand \elt{élé\-ment\xspace}  
\newcommand \elts{élé\-ments\xspace}  

\def \endo {endo\-mor\-phisme\xspace}
\def \endos {endo\-mor\-phismes\xspace}

\newcommand \entrel {rela\-tion impli\-ca\-tive\xspace}
\newcommand \entrels {rela\-tions impli\-ca\-tives\xspace}

\newcommand\evc{espa\-ce vec\-to\-riel\xspace} 
\newcommand\evcs{espa\-ces vec\-to\-riels\xspace} 

\newcommand \eqv {équi\-valent\xspace}  
\newcommand \eqve {équi\-va\-lente\xspace}  
\newcommand \eqvs {équi\-valents\xspace}  
\newcommand \eqves {équi\-val\-entes\xspace}  

\newcommand \eqvc {équi\-va\-lence\xspace}  
\newcommand \eqvcs {équi\-va\-lences\xspace}  

\newcommand \esid {essen\-tiel\-lement iden\-tique\xspace}  
\newcommand \esids {essen\-tiel\-lement iden\-tiques\xspace}  

\newcommand \Esid {\hyperref[defitdyesidentiques]{\esid}}  
\newcommand \Esids {\hyperref[defitdyesidentiques]{\esids}}  

\newcommand \eseq {essen\-tiel\-lement \eqve}  
\newcommand \eseqs {essen\-tiel\-lement \eqves}  

\newcommand \Eseq {\hyperref[defitheseq]{\eseq}}  
\newcommand \Eseqs {\hyperref[defitheseq]{\eseqs}}

%:  f g
\newcommand \fab {\fcn bornée\xspace}
\newcommand \fabs {\fcns bornées\xspace}

\newcommand \fat {\fcn totale\xspace}
\newcommand \fats {\fcn totales\xspace}

\newcommand \fap {\fcn partielle\xspace}
\newcommand \faps {\fcns partielles\xspace}

\newcommand \fip {filtre pre\-mier\xspace}
\newcommand \fips {filtres pre\-miers\xspace}

\newcommand \fipma {\fip maxi\-mal\xspace}
\newcommand \fipmas {\fips maxi\-maux\xspace}

\newcommand \fit {fidè\-lement\xspace}

\newcommand \fcn {factorisation\xspace}
\newcommand \fcns {factorisations\xspace}

\newcommand \fdi {for\-te\-ment dis\-cret\xspace}
\newcommand \fdis {for\-te\-ment dis\-crets\xspace}

\newcommand \fsa {fermé \sagq}
\newcommand \fsas {fermés \sagqs}

\newcommand \fsagc {fonction \sagc}
\newcommand \fsagcs {fonctions \sagcs}

\newcommand \fmt {formel\-lement\xspace}

\newcommand \fpt {\fit plat\xspace}
\newcommand \fpte {\fit plate\xspace}
\newcommand \fpts {\fit plats\xspace}
\newcommand \fptes {\fit plates\xspace}

\newcommand \frl {for\-tement réticulé\xspace}
\newcommand \frle {for\-tement réticulée\xspace}
\newcommand \frls {for\-tement réticulés\xspace}

\newcommand \ftm {fortement\xspace}

\newcommand\gmt{géométrie\xspace}  
\newcommand\gmts{géométries\xspace}  

\newcommand\gaq{\gmt \agq}  

\newcommand\gmq{géomé\-trique\xspace}  
\newcommand\gmqs{géomé\-triques\xspace}  

\newcommand\gmqt{géomé\-tri\-quement\xspace}  

\newcommand\gne{géné\-ra\-lisé\xspace}  
\newcommand\gnee{géné\-ra\-lisée\xspace}  
\newcommand\gnes{géné\-ra\-lisés\xspace}  
\newcommand\gnees{géné\-ra\-lisées\xspace}  

\newcommand\gnl{géné\-ral\xspace}  
\newcommand\gnle{géné\-rale\xspace}  
\newcommand\gnls{géné\-raux\xspace}  
\newcommand\gnles{géné\-rales\xspace}  

\newcommand\gnlt{géné\-ra\-lement\xspace}  

\newcommand\gnn{géné\-ra\-li\-sa\-tion\xspace}  
\newcommand\gnns{géné\-ra\-li\-sa\-tions\xspace}  

\newcommand\gnq {géné\-rique\xspace}  
\newcommand\gnqs {géné\-riques\xspace}  

\newcommand\gnr{géné\-ra\-liser\xspace}  

\newcommand \gns{géné\-ra\-lise\xspace}

\newcommand \gnt{géné\-ra\-lité\xspace}
\newcommand \gnts{géné\-ra\-lités\xspace}

\newcommand \grl{groupe \rtl}
\newcommand \grls{groupes \rtls}

\newcommand \gRl {\hyperref[theorieGrl]{\grl}}
\newcommand \gRls {\hyperref[theorieGrl]{\grls}}

\newcommand\gtr{géné\-ra\-teur\xspace}  
\newcommand\gtrs{géné\-ra\-teurs\xspace}  

%:  h i

\newcommand \homo {ho\-mo\-mor\-phisme\xspace}
\newcommand \homos {ho\-mo\-mor\-phismes\xspace}

\newcommand \hmg {homo\-gène\xspace}
\newcommand \hmgs {homo\-gènes\xspace}

\newcommand \icftr {intervalle compact réticulé \ftm réel\xspace}
\newcommand \icftrs {intervalles compacts réticulés \ftm réels\xspace}

\newcommand \icl {inté\-gra\-lement clos\xspace}
\newcommand \icle {inté\-gra\-lement close\xspace}

\newcommand \icsr {intervalle compact \stm réticulé\xspace}
\newcommand \icsrs {intervalles compacts \stm réticulés\xspace}

\newcommand \icrc {intervalle compact réel clos\xspace}
\newcommand \icrcs {intervalles compact réels clos\xspace}

\newcommand \id {idéal\xspace}
\newcommand \ids {idéaux\xspace}

\newcommand \ida {\idt \agq}
\newcommand \idas {\idts \agqs}

\newcommand \idc  {\idt de Cramer\xspace}
\newcommand \idcs {\idts de Cramer\xspace}

\newcommand \idd {idéal déter\-minan\-tiel\xspace}
\newcommand \idds {idéaux déter\-minan\-tiels\xspace}

\newcommand \idema {idéal maxi\-mal\xspace}
\newcommand \idemas {idéaux maxi\-maux\xspace}

\newcommand \idep {idéal pre\-mier\xspace}
\newcommand \ideps {idéaux pre\-miers\xspace}

\newcommand \idemi {\idep minimal\xspace}
\newcommand \idemis {\ideps minimaux\xspace}

\newcommand \idf {idéal de Fitting\xspace}
\newcommand \idfs {idéaux de Fitting\xspace}

\newcommand \idif {idéal \dvlg fini\xspace}
\newcommand \idifs {idéaux \dvlgs finis\xspace}

\newcommand \idli {idéal \dvli\xspace} %ideal
\newcommand \idlis {idéaux \dvlis\xspace} %ideal

\newcommand \idm {idem\-po\-tent\xspace}
\newcommand \idms {idem\-po\-tents\xspace}
\newcommand \idme {idem\-po\-tente\xspace}
\newcommand \idmes {idem\-po\-tentes\xspace}

\newcommand \idp {idéal prin\-ci\-pal\xspace}
\newcommand \idps {idé\-aux prin\-ci\-paux\xspace}

\newcommand \idt {iden\-ti\-té\xspace}
\newcommand \idts {iden\-ti\-tés\xspace}

\newcommand \idtr {in\-dé\-ter\-mi\-née\xspace}
\newcommand \idtrs {in\-dé\-ter\-mi\-nées\xspace}

\newcommand \ifr {idéal frac\-tion\-nai\-re\xspace}
\newcommand \ifrs {idéaux frac\-tion\-nai\-res\xspace}

\newcommand \imd {immé\-diat\xspace}
\newcommand \imde {immé\-diate\xspace}
\newcommand \imds {immé\-diats\xspace}
\newcommand \imdes {immé\-diates\xspace}

\newcommand \imdt {immé\-dia\-te\-ment\xspace}

\newcommand \indtr {inf-demi-treillis\xspace} 

\newcommand \inteq {intui\-ti\-vement \eqve}
\newcommand \inteqs {intui\-ti\-vement \eqves}

\newcommand \Inteq {\hyperref[defextintequiv]{\inteq}}
\newcommand \Inteqs {\hyperref[defextintequiv]{\inteqs}}

\newcommand \ing {in\-ver\-se \gne}
\newcommand \ings {in\-ver\-ses \gnes}

\newcommand \iMP {in\-ver\-se de Moo\-re-Pen\-ro\-se\xspace}
\newcommand \iMPs {in\-ver\-ses de Moo\-re-Pen\-ro\-se\xspace}

\newcommand \ipp {\idep poten\-tiel\xspace}
\newcommand \ipps {\ideps poten\-tiels\xspace}

\newcommand \ird {irré\-duc\-tible\xspace}
\newcommand \irds {irré\-duc\-tibles\xspace}

\newcommand \irdt {irré\-duc\-ti\-bi\-lité\xspace}

\newcommand \iso {iso\-mor\-phisme\xspace}
\newcommand \isos {iso\-mor\-phismes\xspace}

\newcommand \itf {idéal \tf}
\newcommand \itfs {idé\-aux \tf}

\newcommand \itid {intui\-ti\-vement iden\-tique\xspace}
\newcommand \itids {intui\-ti\-vement iden\-tiques\xspace}

\newcommand \iv {inversible\xspace}
\newcommand \ivs {inversibles\xspace}

\newcommand \ivdg {inverse divisoriel\xspace} %groupes
\newcommand \ivdgs {inverses divisoriels\xspace} %groupes

\newcommand \ivde {inverse divisorielle\xspace} %anneaux, liste 
\newcommand \ivdes {inverses divisorielles\xspace} %anneaux, liste 

\newcommand \ivda {inverse divisoriel\xspace} %anneaux, ideaux
\newcommand \ivdas {inverses divisoriels\xspace} %anneaux

\newcommand \ivmp {il en va de même pour\xspace}

%:  l m

\newcommand \lgb {local-global\xspace}
\newcommand \lgbe {locale-globale\xspace}
\newcommand \lgbs {local-globals\xspace}

\newcommand \lin {liné\-aire\xspace}
\newcommand \lins {liné\-aires\xspace}

\newcommand \lint {liné\-ai\-rement\xspace}

\newcommand \lmo {\lot mono\-gène\xspace}
\newcommand \lmos {\lot mono\-gènes\xspace}

\newcommand \lnl {\lot \nl}
\newcommand \lnls {\lot \nls}

\newcommand \lot {loca\-lement\xspace}

\newcommand \lon {loca\-li\-sation\xspace}
\newcommand \lons {loca\-li\-sations\xspace}

\newcommand \lop {\lot prin\-cipal\xspace}
\newcommand \lops {\lot prin\-cipaux\xspace}

\newcommand \lsdz {\lot \sdz}

\newcommand \mdi {mo\-dule des \diles}

\newcommand \mlm {mo\-dule \lmo}
\newcommand \mlms {mo\-dules \lmos}

\newcommand \mlmo {ma\-tri\-ce de loca\-li\-sation
mono\-gène\xspace}
\newcommand \mlmos {ma\-tri\-ces de loca\-li\-sation
mono\-gène\xspace}

\newcommand \mlp {ma\-tri\-ce de loca\-li\-sation
prin\-ci\-pa\-le\xspace}
\newcommand \mlps {ma\-tri\-ces de loca\-li\-sation
prin\-ci\-pa\-le\xspace}

\newcommand \mo {mo\-no\"{\i}de\xspace}
\newcommand \mos {mo\-no\"{\i}des\xspace}

\newcommand \moco {\mos \com}

\newcommand \molo {morphisme de \lon\xspace}
\newcommand \molos {morphismes de \lon\xspace}

\newcommand \mom {mo\-nô\-me\xspace}
\newcommand \moms {mo\-nô\-mes\xspace}

\newcommand \moquo {morphisme de passage au quotient\xspace}
\newcommand \moquos {morphismes de passage au quotient\xspace}

\newcommand \mpf {mo\-dule \pf}
\newcommand \mpfs {mo\-dules \pf}

\newcommand \mpl {mo\-dule plat\xspace}
\newcommand \mpls {mo\-dules plats\xspace}

\newcommand \mpn {ma\-trice de \pn}
\newcommand \mpns {ma\-trices de \pn}

\newcommand \mprn {ma\-trice de \prn}
\newcommand \mprns {ma\-trices de \prn}

\newcommand \mptf {mo\-dule \ptf}
\newcommand \mptfs {mo\-dules \ptfs}

\newcommand \mrc {mo\-dule \prc}
\newcommand \mrcs {mo\-dules \prcs}

\newcommand \mtf {mo\-du\-le \tf}
\newcommand \mtfs {mo\-du\-les \tf}

%:  n 

\newcommand \ncr{néces\-saire\xspace}       
\newcommand \ncrs{néces\-saires\xspace}       

\newcommand \ncrt{néces\-sai\-rement\xspace}       

\newcommand \ndz {régu\-lier\xspace}
\newcommand \ndzs {régu\-liers\xspace}

\newcommand \nl {simple\xspace}
\newcommand \nls {simples\xspace}

\newcommand \noco {\noe \coh}
\newcommand \nocos {\noes \cohs}

\newcommand \Noe {Noether\xspace}

\newcommand \noe {noethé\-rien\xspace}
\newcommand \noes {noethé\-riens\xspace}
\newcommand \noee {noethé\-rienne\xspace}
\newcommand \noees {noethé\-riennes\xspace}

\newcommand \noet {noethé\-ria\-nité\xspace}

\newcommand \nst {Null\-stellen\-satz\xspace}
\newcommand \nsts {Null\-stellen\-s\"atze\xspace}

%:  o
\newcommand \op{opé\-ra\-tion\xspace}  
\newcommand \ops{opé\-ra\-tions\xspace}
\newcommand \opari{\op\ari}  
\newcommand \oparis{\ops\aris}  
\newcommand \oparisv{\ops\arisv}  

\newcommand \oqc {ouvert \qc}
\newcommand \oqcs {ouverts \qcs}

\newcommand \ort{or\-tho\-go\-nal\xspace}  
\newcommand \orte{or\-tho\-go\-na\-le\xspace}  
\newcommand \orts{or\-tho\-go\-naux\xspace}  
\newcommand \ortes{or\-tho\-go\-na\-les\xspace}  

%:  p

\newcommand \pa {couple saturé\xspace}
\newcommand \pas {couples saturés\xspace}
 
\newcommand \paral{paral\-lèle\xspace}  
\newcommand \parals{paal\-lèles\xspace}  

\newcommand \paralm{paral\-lè\-lement\xspace}

\newcommand \pb{pro\-blè\-me\xspace}  
\newcommand \pbs{pro\-blè\-mes\xspace}  

\newcommand \peq {purement équa\-tion\-nelle\xspace}
\newcommand \peqs {purement équa\-tion\-nelles\xspace}

\newcommand \pf {de \pn finie\xspace}

\newcommand \pgn {polygone de Newton\xspace}
\newcommand \pgns {polygones de Newton\xspace}

\newcommand \plc {rési\-duel\-lement \zed}
\newcommand \plcs {rési\-duel\-lement \zeds}

\newcommand \Plg {Prin\-cipe \lgb}
\newcommand \plg {prin\-cipe \lgb}
\newcommand \plgs {prin\-cipes \lgbs}

\newcommand \plga {\plg abs\-trait\xspace}
\newcommand \plgas {\plgs abs\-traits\xspace}

\newcommand \Plgc {\Plg con\-cret\xspace}
\newcommand \plgc {\plg con\-cret\xspace}
\newcommand \plgcs {\plgs con\-crets\xspace}

\newcommand \pn {présen\-ta\-tion\xspace}
\newcommand \pns {présen\-ta\-tions\xspace}

\newcommand \pog {\pol \hmg\xspace}
\newcommand \pogs {\pols \hmgs\xspace}

\newcommand \Pol {Poly\-nôme\xspace}
\newcommand \Pols {Poly\-nômes\xspace}

\newcommand \pol {poly\-nôme\xspace}
\newcommand \pols {poly\-nômes\xspace}

\newcommand \poll{poly\-nomial\xspace}  
\newcommand \polls{poly\-nomiaux\xspace}  
\newcommand \polle{poly\-no\-miale\xspace}  
\newcommand \polles{poly\-no\-miales\xspace}  

\newcommand \pollt{poly\-no\-mia\-lement\xspace}  

\newcommand \polfon {\pol fon\-da\-men\-tal\xspace}
\newcommand \polmu {\pol rang\xspace}
\newcommand \polmus {\pols rang\xspace}
\newcommand \polcar {\pol carac\-té\-ris\-tique\xspace}
\newcommand \polcars {\pols carac\-té\-ris\-tiques\xspace}
\newcommand \polmin {\pol mini\-mal\xspace}

\newcommand \prc {\pro de rang constant\xspace}
\newcommand \prcs {\pros de rang constant\xspace}

\newcommand \prcc {prin\-ci\-pe de \rcc}
\newcommand \prca {prin\-ci\-pe de \rca}
\newcommand \prce {prin\-ci\-pe de \rce}

\newcommand \prmt {préci\-sé\-ment\xspace}
\newcommand \Prmt {Préci\-sé\-ment\xspace}

\newcommand \prn {pro\-jec\-tion\xspace}
\newcommand \prns {pro\-jec\-tions\xspace}

\newcommand \pro {pro\-jec\-tif\xspace}
\newcommand \pros {pro\-jec\-tifs\xspace}

\newcommand \prr {pro\-jec\-teur\xspace}
\newcommand \prrs {pro\-jec\-teurs\xspace}

\newcommand \Prt {Pro\-pri\-été\xspace}
\newcommand \Prts {Pro\-pri\-étés\xspace}
\newcommand \prt {pro\-pri\-été\xspace}
\newcommand \prts {pro\-pri\-étés\xspace}

\newcommand \ptf {\pro \tf}
\newcommand \ptfs {\pros \tf}

\newcommand \qc {quasi-compact\xspace}
\newcommand \qcs {quasi-compacts\xspace}

\newcommand \qi {qua\-si in\-tè\-gre\xspace}
\newcommand \qis {qua\-si in\-tè\-gres\xspace}

\newcommand \qnl {quasi-\nl}
\newcommand \qnls {quasi-\nls}

%:  r
\newcommand \ralg {règle \agq}
\newcommand \ralgs {règles \agqs}

\newcommand \rav {racine virtuelle\xspace}
\newcommand \ravs {racines virtuelles\xspace}

\newcommand \rcc {\rcm con\-cret\xspace}
\newcommand \rca {\rcm abs\-trait\xspace}
\newcommand \rce {\rcc des é\-ga\-li\-tés\xspace}

\newcommand \rcm {recol\-lement\xspace}
\newcommand \rcms {recol\-lements\xspace}

\newcommand \rcv {recou\-vrement\xspace} 
\newcommand \rcvs {recou\-vrements\xspace}

\newcommand \rde {rela\-tion de dépen\-dance\xspace}
\newcommand \rdes {rela\-tions de dépen\-dance\xspace}

\newcommand \rdi {\rde inté\-grale\xspace}
\newcommand \rdis {\rdes inté\-grales\xspace}

\newcommand \rdl {\rde \lin}
\newcommand \rdls {\rdes \lins}

\newcommand \rdt {rési\-duel\-lement\xspace}

\newcommand \rdy {règle dyna\-mique\xspace}
\newcommand \rdys {règles dyna\-miques\xspace}

\newcommand \red {règle directe\xspace}
\newcommand \reds {règles directes\xspace}

%{\hyperref[defithconserv]{conser\-vative\xspace}}
\newcommand \rex {\hyperref[defexistsimple]{règle exis\-ten\-tielle simple\xspace}}
\newcommand \rexs {\hyperref[defexistsimple]{règles exis\-ten\-tielles simples\xspace}}

\newcommand \rexri {\hyperref[defitdyexrig]{règle exis\-ten\-tielle rigide\xspace}}
\newcommand \rexris {\hyperref[defitdyexrig]{règles exis\-ten\-tielles rigides\xspace}}

\newcommand \rsim {règle de simplification\xspace}
\newcommand \rsims {règles de simplification\xspace}

\newcommand \rtl {réti\-culé\xspace}
\newcommand \rtls {réti\-culés\xspace}

\newcommand \rmq {\rcm de quotients\xspace} % recollement de quotients
\newcommand \rvq {\rcv par quotients\xspace} % recouvrement par quotients
\newcommand \rmqs {\rcms de quotients\xspace} % 
\newcommand \rvqs {\rcvs par quotients\xspace} % 

\newcommand \rpf {réduite-de-présen\-tation-finie\xspace}
\newcommand \rpfs {réduites-de-présen\-tation-finie\xspace}

%:  s

\newcommand \sad {\salg dynamique\xspace}
\newcommand \sads {\salgs dynamiques\xspace}

\newcommand \sagq {semi\agq}
\newcommand \sagqs {semi\agqs}

\newcommand \sagc {\sagq continue\xspace}
\newcommand \sagcs {\sagqs continues\xspace}

\newcommand \salg {structure \agq}
\newcommand \salgs {structures \agqs}

\newcommand \scentrel {relation semi-implicative\xspace}
\newcommand \scentrels {relations semi-implicatives\xspace}

\newcommand \scf {schéma fini\-taire\xspace}
\newcommand \scfs {schémas fini\-taires\xspace}

\newcommand \scl {schéma \elr}
\newcommand \scls {schémas \elrs}

\newcommand \sdo {\sdr \orte}
\newcommand \sdos {\sdrs \ortes}

\newcommand \sdr {somme directe\xspace}
\newcommand \sdrs {sommes directes\xspace}

\newcommand \sdz {sans \dvz}

\newcommand \sfio {sys\-tème fondamental d'\idms ortho\-gonaux\xspace}
\newcommand \sfios {sys\-tèmes fondamentaux d'\idms ortho\-gonaux\xspace}

\newcommand \sgr {\sys \gtr}
\newcommand \sgrs {\syss \gtrs}

\newcommand \slgb {stricte\-ment \lgb}
\newcommand \slgbs {stricte\-ment \lgbs}

\newcommand \sli {\sys \lin}
\newcommand \slis {\syss \lins}

\newcommand \smq {symé\-trique\xspace}
\newcommand \smqs {symé\-triques\xspace}

\newcommand \spb {sépa\-rable\xspace}  % algebres
\newcommand \spbs {sépa\-rables\xspace}

\newcommand \spe {spéci\-fi\-cation\xspace}
\newcommand \spes {spéci\-fi\-cations\xspace}

\newcommand \spi {\spe incomplète\xspace}
\newcommand \spis {\spes incomplètes\xspace}

\newcommand \spl {sépa\-rable\xspace}  % polynomes
\newcommand \spls {sépa\-rables\xspace}

\newcommand \spo {semipolynôme\xspace}
\newcommand \spos {semipolynômes\xspace}

\newcommand \spt{sépa\-ra\-bi\-lité\xspace}

\newcommand \srg {suite régu\-lière\xspace}
\newcommand \srgs {suites régu\-lières\xspace}

\newcommand \ste {strictement étale\xspace}
\newcommand \stes {strictement étales\xspace}

\newcommand \stf {strictement fini\xspace}
\newcommand \stfs {strictement finis\xspace}
\newcommand \stfe {strictement finie\xspace}
\newcommand \stfes {strictement finies\xspace}

\newcommand \stl {stablement libre\xspace}
\newcommand \stls {stablement libres\xspace}

\newcommand \stm {strictement\xspace}

\newcommand \str {\stm réticulé\xspace}
\newcommand \stre {\stm réticulée\xspace}
\newcommand \strs {\stm réticulés\xspace}
\newcommand \stres {\stm réticulées\xspace}

\newcommand \sul {supplé\-men\-taire\xspace}
\newcommand \suls {supplé\-men\-taires\xspace}

\newcommand \Sut {Support\xspace}
\newcommand \Suts {Supports\xspace}
\newcommand \sut {support\xspace}

\newcommand \syc {\sys de coordon\-nées\xspace}
\newcommand \sycs {\syss de coordon\-nées\xspace}

\newcommand \syp {\sys \poll}
\newcommand \syps {\syss \polls}

\newcommand \sys {sys\-tème\xspace}
\newcommand \syss {sys\-tèmes\xspace}

%:  t
\newcommand \talg {théorie \agq}
\newcommand \talgs {théories \agqs}

\newcommand \tco {théorie cohé\-rente\xspace}
\newcommand \tcos {théories cohé\-rentes\xspace}

\newcommand \tdy {théorie dyna\-mique\xspace}
\newcommand \tdys {théories dyna\-miques\xspace}

\newcommand \tel {\hyperref[defexistsimple]{théorie exis\-ten\-tielle\xspace}}
\newcommand \tels {\hyperref[defexistsimple]{théories exis\-ten\-tielles\xspace}}

\newcommand \telri {\hyperref[defitdyexrig]{théorie exis\-ten\-tielle rigide\xspace}}
\newcommand \telris {\hyperref[defitdyexrig]{théories exis\-ten\-tielles rigides\xspace}}

\newcommand \tf {de type fini\xspace}

\newcommand \tfo {théorie formelle\xspace}
\newcommand \tfos {théorie formelles\xspace}

\newcommand \tgm {théorie \gmq}
\newcommand \tgms {théories \gmqs}

\newcommand \Tho {Théo\-rème\xspace}
\newcommand \Thos {Théo\-rèmes\xspace}
\newcommand \tho {théo\-rème\xspace}
\newcommand \thos {théo\-rèmes\xspace}

\newcommand \thoc {théo\-rème$\mathbf{^*}$~}

\newcommand \tpe {théorie \peq}
\newcommand \tpes {théories \peqs}

\newcommand \trdi {treil\-lis dis\-tri\-bu\-tif\xspace}
\newcommand \trdis {treil\-lis dis\-tri\-bu\-tifs\xspace}

\newcommand \trel {trans\-for\-mation \elr}
\newcommand \trels {trans\-for\-mations \elrs}

%:  u

\newcommand \ultm {ultramétrique\xspace}
\newcommand \ultms {ultramétriques\xspace}

\newcommand \umd {unimo\-du\-laire\xspace}
\newcommand \umds {unimo\-du\-laires\xspace}

\newcommand \unt {uni\-taire\xspace}
\newcommand \unts {uni\-taires\xspace}

\newcommand \uvl {uni\-versel\xspace}
\newcommand \uvle {uni\-ver\-selle\xspace}
\newcommand \uvls {uni\-versels\xspace}
\newcommand \uvles {uni\-ver\-selles\xspace}

%:  v

\newcommand \vala {valeur absolue\xspace}
\newcommand \valas {valeurs absolues\xspace}

\newcommand \vfn {véri\-fi\-cation\xspace}
\newcommand \vfns {véri\-fi\-cations\xspace}

\newcommand \vmd {vec\-teur \umd}
\newcommand \vmds {vec\-teurs \umds}

\newcommand \zed {z\'{e}ro-di\-men\-sionnel\xspace}
\newcommand \zede {z\'{e}ro-di\-men\-sion\-nelle\xspace}
\newcommand \zeds {z\'{e}ro-di\-men\-sion\-nels\xspace}
\newcommand \zedes {z\'{e}ro-di\-men\-sion\-nelles\xspace}

\newcommand \zedr {\zed réduit\xspace}
\newcommand \zedrs {\zeds réduits\xspace}

\newcommand \zmt {\tho de Zariski-Grothen\-dieck\xspace}

%:  ------- maths constructives

\newcommand \cof {cons\-truc\-tif\xspace}
\newcommand \cofs {cons\-truc\-tifs\xspace}

\newcommand \cov {cons\-truc\-tive\xspace}
\newcommand \covs {cons\-truc\-tives\xspace}

\newcommand \coma {\maths\covs}
\newcommand \clama {\maths clas\-siques\xspace}

\renewcommand \cot {cons\-truc\-ti\-vement\xspace}

\newcommand \matn {mathé\-ma\-ticien\xspace}
\newcommand \matne {mathé\-ma\-ti\-cienne\xspace}
\newcommand \matns {mathé\-ma\-ticiens\xspace}
\newcommand \matnes {mathé\-ma\-ti\-ciennes\xspace}

\newcommand \maths {mathé\-ma\-tiques\xspace}
\newcommand \mathe {mathé\-ma\-tique\xspace}

\newcommand \prco {démons\-tration \cov}
\newcommand \prcos {démons\-trations \covs}